\documentclass[11pt]{article}
\usepackage[a4paper,total={6in,8in}]{geometry}
\usepackage[T1]{fontenc}
\usepackage[utf8]{inputenc}
\usepackage{authblk}
\usepackage{enumitem}
\usepackage{amsmath}
\usepackage{amsthm}
\usepackage{amssymb}
\usepackage[all]{xy}
\usepackage{stmaryrd}
\usepackage[toc,page,title]{appendix}

\newcommand{\mrm}{\mathrm}
\newcommand{\eff}{{\mathrm{eff}}}

\newcommand\numberthis{\addtocounter{equation}{1}\tag{\theequation}}

\newcommand\Pa{\mathcal{P} }
\newcommand\La{\mathcal{L} }
\newcommand\R{\mathbb{R} }
\newcommand\N{\mathbb{N} }

\newcommand\Z{\mathbb{Z} }

\newcommand\one{\mathbb{1}}

\newcommand{\E}{\mathbb{E}}
\newcommand{\J}{\mathbb{J}}

\renewcommand{\S}{\mathbb{S}}
\renewcommand{\P}{\mathbb{P}}
\renewcommand{\L}{\mathbb{L}}
\newcommand{\D}{\mathbb{D}}

\newcommand{\calL}{\mathcal{L}}
\newcommand{\calP}{\mathcal{P}}
\newcommand{\calM}{\mathcal{M}}
\newcommand{\calD}{\mathcal{D}}

\newcommand{\calB}{\mathcal{B}}

\newcommand{\calK}{\mathcal{K}}

\newcommand{\eps}{\varepsilon}
\newcommand{\ph}{\varphi}

\newcommand{\Id}{\mathrm{Id}}
\DeclareMathOperator{\Law}{Law} 

\newcommand{\Proba}{\mathrm{Pr}}

\newcommand{\e}{{\rm e}}
\renewcommand{\d}{\,\mathrm{d}}

\newcommand{\eqdef}{ \dps \mathop{=}^{{\rm def}} }
\newcommand{\dps}{\displaystyle}

\newcommand{\abs}[1]{\left | #1\right |}
\newcommand{\set}[1]{\left\{#1\right\}}
\newcommand{\p}[1]{ \left(#1\right) }
\renewcommand{\b}[1]{ \left[#1\right] }

\newcommand{\norm}[1]{\left\Vert#1\right\Vert}

\renewcommand{\le}{\leqslant}
\renewcommand{\ge}{\geqslant}
 
\theoremstyle{plain}
\newtheorem{The}{Theorem}[section]
\newtheorem{Lem}[The]{Lemma}
\newtheorem{Pro}[The]{Proposition}

\newtheorem{Cor}[The]{Corollary}
\newtheorem{Def}[The]{Definition}
\newtheorem{Ass}{Assumption}

\newtheorem{Rem}[The]{Remark}

\title{A general martingale approach to large noise homogenization}

\begin{document}
\author[1]{Dimitri Faure}
\author[2]{Mathias Rousset}
\affil[1]{DMA, École normale supérieure, Université PSL, CNRS, 75005 Paris, France}
\affil[2]{IRMAR and Inria, University of Rennes, France.}
\maketitle
\abstract{We consider a family of Markov processes with generator of the form $\gamma \calL_{1} + \calL_{0}$, in which $\calL_{1}$ generates a so-called \emph{dominant process} The dominant process is assumed to converge at large times towards a random point in a given subset of the state space called the \emph{effective state space}. Using the usual characterization through martingale problems, we give general conditions under which homogenization holds true: the original process converges, when $\gamma$ is large and for both the Meyer-Zheng pseudo-path topology and finite-dimensional time marginals, towards an identified effective Markov process on the effective space. Few simple model examples for diffusions are studied.
}

\tableofcontents

\section*{Introduction}

\paragraph{A formalism for large noise homogenization.} Let $E$ be a Polish state space, and let $E_{\eff} \subset E$ be a closed subset, hereafter called the \emph{effective state space}. This paper considers on $E$ a family of (c\`adl\`ag) Markov process denoted $\p{X_t^{\gamma}}_{t \geq 0}\equiv \p{X_t^{(\gamma,1)}}_{t \geq 0}$ with (formal) infinitesimal Markov generator of the form
\begin{equation}\label{eq:main_gen} 
\gamma \calL_{1} + \calL_{0},
\end{equation}
in which $\calL_{1},\calL_{0}$ are two (formal) infinitesimal Markov generators, and $\gamma > 0$ is a positive parameter that is meant to go to infinity ($\gamma \to +\infty$). 

In this setting, the following main structural assumption lay the basis to what is here called \emph{large noise homogenization}. For each initial condition $x \in E$, the process generated by $\calL_{1}$ (called here after the \emph{dominant} process/generator) and denoted $\p{X_t^{(1,0)}(x)}_{t \geq 0}$, is assumed to converge in probability for large time towards a random point in $E_\eff$:
\begin{equation}\label{eq:proj}
\begin{cases}
&\lim_{t \to +\infty} X^{(1,0)}_t(x) =  X^{(1,0)}_\infty(x) \in E_\eff \quad \text{in probability,}\\
& \Pa(x, \, .\,) \eqdef \Law(X^{(1,0)}_\infty(x)).
\end{cases}
\end{equation} 
One also want to define $E_\eff$ as the smallest effective space satisfying the above, which suggests the following associated additional assumption 
\begin{equation}\label{eq:minim}
\calP(x, \, . \,) = \delta_x, \qquad \forall x \in E_\eff .
\end{equation}
It also implies by definition of $\Pa$ that $\calL_{1} \Pa \ph = 0$ for any appropriately regular bounded test function $\ph: E_{\eff} \to \R$ on the effective state space. One can then define the operator
\begin{equation}\label{eq:eff}
\calL_{\eff} \ph \eqdef \calL_{0} \Pa \ph = \p{\gamma \calL_{1} + \calL_{0}}\Pa \ph,
\end{equation}
which is formally a Markov generator on $E_{\eff}$ as a limit when  $s \to 0$ of a pure jump Markov generator on $E_{\eff}$ given by 
$$\calL_{\eff} = \lim_{s \to 0} \p{ \e^{ s \calL_{0}} \Pa - \Id }/s , $$
where we have used the 'minimality' property~\eqref{eq:minim} of $E_\eff$. The process
\begin{equation} \label{eq:mart_0}
t \mapsto \Pa \ph\p{X_t^{\gamma}} - \int_0^t \calL_{0} \Pa \ph (X_s^{\gamma}) \d s
\end{equation}
is a martingale for the natural filtration generated by $X^\gamma$ for a large class of regular bounded $\ph$. Since by~\eqref{eq:proj} the dominant generator $\calL_{1}$ tends to confine the process $X^\gamma$ near $E_{\eff}$ when $\gamma$ is large, one may then expect in that asymptotics that: i) $X^\gamma$ behave approximately as a Markov process taking values in $E_\eff$ and, ii) by~\eqref{eq:mart_0}, it becomes a solution to the martingale problem in $E_\eff$ associated with $\calL_{\eff}$. Making this last statement rigorous and identifying the effective process is precisely what we call here the \emph{martingale problem approach to strong noise homogenization}. This will be carried out in the present work in a quite general setting using (as already proposed in previous works mentioned below) the pseudo-path topology studied in the seminal work~\cite{MZ84} of Meyer and Zheng. This should be contrasted with \textit{e.g.}~\cite{k12}, a classical reference on the martingale problem approach in a classical homogenization setting, which is using more classically the Skorokhod ('uniform') topology. 

Note that, in strong noise problems, if the initial process is, say, a diffusion, then the effective process is typically of jump type, so that convergence with respect to the stronger Skorokhod topology (instead of the pseudo-path topology) cannot be true.

\paragraph{A class of examples.} The main examples studied in this work (Section~\ref{sec:examples}) will be of the following form: $E$ will be a closed subset of a manifold (say $\R^d$) with a piece-wise smooth boundary; and $E_{\eff}$ will be a sub-manifold of this boundary (in this paper either a finite subset of the boundary, or the full boundary). The dominant process will be a diffusion which eventually almost surely hits the boundary of $E$ at $E_{\eff}$ (here with an infinite hitting time); and the perturbation process will be deterministic flow with drift pointing inward at the boundary of $E$. Note that in the context of potential theory, if the dominant process is a (singularly at $E_\eff$) time-changed Brownian motion and $E_{\eff} = \partial E$ is the boundary of a smooth domain $E$, then the kernel $\Pa(x, \, . \,)$ is exactly the so-called \emph{harmonic measure} of $E$ with pole at $x$. When $\calL_{0}$ is an inward vector field, the Markovian effective generator~\eqref{eq:eff} is also known as a quite standard general version of the so-called \emph{Dirichlet-to-Neumann} operator. The latter has been extensively studied in the inverse problem literature (see \textit{e.g.} the seminal work~\cite{sylvester1987global}) for reasons of independent interest. The Dirichlet-to-Neumann effective operator is, in fact, a non-diffusive Markov operator on $E_\eff$ of Levy type, as will be demonstrated in the examples of Section~\ref{sec:examples}. More comments on that topic will be given in the end of this introduction. 

\paragraph{Previous work.} Motivated by quantum continuous measurements (see~\cite{q1,q2,q3,q4}), or filtering theory (see~\cite{cedric2022spike}), strong noise homogenization has recently been rigorously studied in the special case where $E_{\eff}$ is a finite set. In \cite{BBCCNP}, the authors show strong noise homogenization for a certain class of diffusion processes having a generator of the form\footnote{More precisely, in~\cite{BBCCNP} a very slightly more general situation is considered with the possible additional presence of generators at intermediate time scales (in terms of $\gamma$).}~\eqref{eq:main_gen}. In order to prove their theorem, the authors develop a {\it{finite-dimensional}} homogenization theorem~\cite[Theorem 3.1]{BBCCNP} for {\it{bounded operators}} and, due to the specific form of the diffusion processes considered (linear drift and quadratic mobility), are able to prove the convergence by developing a perturbative argument. In~\cite{faure2022averaging}, the author adapts the proof of  \cite[Theorem 3.1]{BBCCNP} for some (unbounded) generators of diffusion operators living in a simplex $K\subset {\mathbb R}^n$ whose vertices define $E_{\eff}$, together with a martingale dominant process. This ensures that the kernel of ${\mathcal L}_{(1)}$ is a finite-dimensional subspace of affine functions of $\R^d$ which simplifies the analysis. 

 The above works consider the so-called (Meyer-Zheng) pseudo-path topology (see~\cite{MZ84} and Section~\ref{sec:MeyerZheng}) on trajectories, as well as the associated Meyer-Zheng ''tightness'' criteria, in order to handle strong noise homogenization. 
 
 The use of the Meyer-Zheng criteria in a homogenization context is however neither new nor limited to the recent consideration of strong noise problems. It has been quite consistently used more than a decade ago to handle other intricate cases; for instance for periodic problems (see~\cite{pardoux2}), or for backward stochastic differential equations (see \cite{pardoux1}).
 
\paragraph{New results.} By considering the classical martingale problem approach, this paper generalizes the previously mentioned body of work. The basic idea consists in remarking that the pseudo-path topology and the associated Meyer-Zheng criteria are well-suited to martingale problems, which is a standard tool for weak (in law) convergence results in classical homogenization theory. Note that in classical references (as~\cite{k12}), martingale problems and limits of processes are handled with the Skorokhod topology (see~\cite{EK} for a general reference on convergence of processes and martingale problems using this topology). 

In the present work, the homogenization scheme consists in proving the following steps.

First, i) relative compactness of trajectory distributions is, by definition of the pseudo-path topology, a usually easy consequence of Prokhorov theorem. One can then consider extracted converging sequences. 

Second, ii) one needs to check that the limiting pseudo-path takes its values in $E_{\eff}$. This step may be quite technical, and we propose a relatively natural setting that enables to carry out Step~ii) by first proving an intermediate property: the dominant process converges towards $E_\eff$ for large times uniformly with respect to initial conditions in compacts (see Proposition~\ref{cvunif}). 

Third, iii) c\`adl\`ag martingales constructed from the full process~\eqref{eq:main_gen} with specific test functions of the form $\Pa \ph$ are expected to have by~\eqref{eq:eff} a mean variation uniformly bounded in $\gamma$: this is the so-called Meyer-Zheng criteria. This criteria (a key result of~\cite{MZ84}) typically implies that the considered martingales do converge to c\`adl\`ag limiting martingales. The limit of the extracted sequence of processes in step i) is then expected to satisfy the martingale problem associated with $\calL_{\eff}$. 

Finally, iv) one simply needs to assume that the latter effective martingale problem is well-posed (uniqueness in law of Markov processes solution of it). We thus eventually obtain convergence in distribution of the considered process towards the effective process in the pseudo-path topology. Extension of the convergence to finite-dimensional time-marginals are also obtained.

In order to make the above four steps i), ii), iii) and iv) rigorous, two assumptions are key. The first assumption (later on denoted Assumption~\ref{ass:P_cont}) is related to the continuity of the dominant process hitting distribution $\Pa(x, \, .\,)$ with respect to the initial condition $x$. This assumption is key in proving the large time convergence of the dominant process uniformly with respect to initial conditions, and, subsequently Step~ii). It is inherited from~\cite{faure2022averaging} in which it is argued that it is somehow unavoidable using a counter-example where the limit process is not Feller. The second important assumption is of course the fact that the effective generator $\calL_\eff$ is well-defined, continuous near $E_\eff$ and defines a well-posedness martingale problem. Well-posedness of martingale problems may not be so easy to obtain, but it is a topic on its own (well covered at least in $\R^d$) in the literature. When the effective dynamics can be described by a stochastic differential equation which is itself well-posed in the strong sense (Lipschitz coefficients), quite generic results exist (see \textit{e.g.} the reference~\cite{kurtz2011}).

\paragraph{A simple example of a new limit theorem.}
As an example of the novelty of the obtained results, one may consult the example of Section~\ref{sec:exa_strip}. In the latter, one considers an already studied one dimensional toy model for strong noise homogenization which consists of a dominant time-changed Brownian motion in $[-1,1]$ perturbed by a drift. Previous works have shown the convergence towards a jump process in $E_{\eff}=\set{+1,-1}$. The novelty consist in enriching the process with its quadratic variation, leading to a process defined on $E=[-1,1] \times \R_+$. This enrichment may be thought of as a way to keep track of the ''fast time scale'' of the dominant process. We prove convergence of the enriched process to a limiting effective process in $\{-1,+1\} \times \R_+$. The latter is the sum of two generators.

The first is a jump generator hoping between $+1$ and $-1$ with a random psoitive jump in $\R_+$ (whose distribution is proportional to an explicit finite measure denoted $\nu$).

The second is a Levy generator in $\R_+$ of subordinator type with explicit Levy measure proportional to:
$$\mu(\mathrm{d}h) \eqdef \left(\frac{\pi^2}{16} \displaystyle\sum_{k=-\infty}^{+\infty} k^2 e^{-h\frac{\pi^2 k^2}{8}} \right) \mathrm{d}h.
$$
The jumps of this subordinator is a mixture exponentially distributed jumps of size $\frac{8}{\pi^2 k^2}$ with $k \in \Z$. The latter can be interpreted as the ''sizes'' (in terms of a re-scaled time) of the ''spikes'' of the original process inside the interval $]-1,+1[$ in between a hoping event from $\pm 1$ to $\mp 1$. This phenomenology of ''spikes'' in strong noise homogenization has already been pointed out and discussed, for instance in~\cite{BBCCNP}.

\paragraph{Link with the boundary process associated with reflected diffusions.} Finally, we want to suggest some potentially interesting connections between strong noise homogenization and reflected diffusions defined as solutions of the Skorokhod reflection problem (see \textit{e.g.}~\cite{pilipenko2014introduction} for a general reference on the topic). The context is the following: the dominant process is now a diffusion with generator $\widetilde \calL_{1}$ which hits $E_\eff$ at a finite (as opposed to typically infinite) time. $\widetilde \calL_{1} = a \times \calL_{1}$ can typically be obtained from one another using a time-change scalar field $a>0$ singular near $E_\eff$. In addition $E_{\eff}= \partial E$ is assumed to be the smooth boundary of regular domain $E$ of, say, $\R^n$, and the perturbation generator $\calL_{0}= F_{(0)}^i \partial_i$ is a vector field which is inward-pointing at the boundary $\partial E$. We have already mentioned above that in this context the effective process generator~\eqref{eq:eff} is the so-called \emph{Dirichlet-To-Neumann} (Markov) operator defined on smooth functions of $E_{\eff}$, which has been studied in the inverse problem literature for linear elliptic partial differential equations (see \textit{e.g.} the seminal work~\cite{sylvester1987global}). It is quite interesting to remark that in~\cite{hsu1986excursions}, the author gives a probabilistic (and completely different) interpretation of $\calL_{\eff}$ using reflected diffusions.   

Let us briefly describe the idea. Consider the diffusion constrained in $E$ by the reflection at the boundary $\partial E$ of $E$ generated by the inward vector field $F_{(0)}$. The latter (see~\cite{pilipenko2014introduction}), denoted here $(X_t)_{t \geq 0}$, is classically defined as the solution of the so-called Skorokhod problem, which is obtained by adding a term to a stochastic differential equation satisfied by the non-reflected diffusion. This additional term must of the form  $F_{(0)}(X_t) \d L_t $ where $L_t$ is a non-decreasing adapted process that is constant outside the boundary (it must satisfies $\d L_t = \one_{X_t \in \partial E} \d L_t$); $L_t$ is a kind of local time that has to be determined as an output of the Skorokhod problem. The \emph{boundary process} is then the Markov process $X^{\partial E}_l  = X_{\tau_l}$ obtained by indexing the initial process with the local time $\tau_l = \inf( t: \, L_t=l)$. If $\Pa(x, \, . \,)$ denotes the hitting distribution at the boundary of $E$ of the underlying non-reflected diffusion $\calL_{1}$, it can be shown using a simple application of It\^o formula to $\Pa(\ph)$ for $\ph$ smooth (see~\cite{hsu1986excursions} for the special case of Brownian motion and normal reflection) that the boundary process $(X_{\tau_l})_{l \geq 0}$ is precisely a solution of the martingale problem associated with $L_{\eff} = \calL_{0} \Pa $. The latter can thus be interpreted as the associated infinitesimal generator.

Developing the connections between strong noise homogenization and boundary processes lies outside the scope of this paper and his left as an open problem of interest. We simply want to mention that reflected diffusions may also be obtained as a limit of diffusions of the form~\eqref{eq:main_gen} using time-changes of the dominant and/or the perturbation process. It is thus not completely surprising to obtain the same 'trace' at the boundary.

\paragraph{Organization.} This work is divided as follows. In Section~\ref{sec:setting}, we rigorously introduce notation, then state the different assumptions, and finally state the main strong noise homogenization theorems. In Section~\ref{sec:proofs}, the main theorems are proved. Finally in Section~\ref{sec:examples}, explicit examples are described and are treated in full details.

\section{Setting, assumptions and main theorem} \label{sec:setting}

\subsection{Notation}
\begin{itemize}
    \item $(X^{(\gamma_1,\gamma_0)}_t)_{t \geq 0}$ the Markov process with generator $\gamma_1 \calL_{1} + \gamma_0 \calL_{0}$.
    \item Short-hand notation:
$
X^\gamma \eqdef X^{(\gamma,1)}.
$
    \item $\mu_0^\gamma  \eqdef \Law(X^\gamma_0)$ an initial distribution.
    \item $t \mapsto X_t(x)$ in order to stress that the initial condition is given by $X_0(x) = x$;
    \item $\E_x$ the expectation over the distribution $(X_t(x))_{t \geq 0}$, when $x \in E$;
    \item $\E_\mu$ the expectation over the distribution $(X_t)_{t \geq 0}$, where $\Law(X_0) = \mu $.
    \item  $L^0(\e^{-t} \d t, E)$: equivalence class of measurable trajectories $\R_+ \to E$ that are Lebesgue almost everywhere equal, and endowed with the Polish pseudo-path topology.
    \item $\Proba(E)$: space of probabilities on $E$.
\end{itemize}

\subsection{Setting and Assumptions}
Let $E$ denote a Polish state space. For each non-negative $\gamma_1,\gamma_0 \geq 0$, we assume that we can measurably associate to any initial condition $x\in E$ a c\`adl\`ag Markov process denoted
$$t \mapsto X_t^{(\gamma_1,\gamma_0)}(x),$$ 
whose (formal) infinitesimal generator is of the form
$$ \calL^{(\gamma_1,\gamma_0)} \eqdef  \gamma_1 \calL_{1} + \gamma_0 \calL_{0}. $$
In the above, $\calL_{1}$ (resp. $\calL_{0}$) are the generator associated with the so-called \emph{dominant} $X^{(1,0)}$ (resp. \emph{perturbation} $X^{(0,1)}$) process. 

In what follows, we denote by $$\mu_0^\gamma = \Law(X_0^\gamma)\in \Proba(E)$$ the initial condition of the considered process. Throughout this paper we will assume that $\mu_0^\gamma$ converges to a limit $\mu_0$ in $\Proba(E)$. 

The main structural assumption underlying the present work is the following: for any $x \in E$, the dominant process $X_t^{(1,0)}(x)$ is assumed to converge for large times $t$ (at least in probability, although the almost sure convergence is typical) towards a random variable taking values in a closed subset $E_{\eff} \subset E$.
\begin{Ass}[Main assumption]\label{ass:main} 
For each $x  \in E$, there is a random variable taking values in $E_\eff$ denoted $X^{(1,0)}_\infty(x) \in E_\eff$ such that, in probability,
$$
\lim_{t \to +\infty} X^{(1,0)}_t(x) = X^{(1,0)}_\infty(x) \in E_\eff.
$$
Moreover, $E_\eff$ is minimal in the sense that if $x \in E_\eff$, then $X_\infty(x) = x$.

\end{Ass}
The limiting distribution thus defines a measurable probability kernel denoted
$$
\Pa(x, \, . \,) \eqdef  \Law(X^{(1,0)}_\infty(x)) \in \Proba(E_\eff).
$$

As explained in the introduction, the main goal of this work is to study the limit effective distribution of the process $X^{(\gamma,1)}$ when $\gamma \to +\infty$, and to show that the latter is a Markov process on $E_\eff $ with infinitesimal generator given by
$$
\calL_{\mrm{eff}} \b{\ph} \eqdef \calL_{0} \b{ \Pa \b{\ph}}.
$$

\begin{Rem}[On invariances by time changes]
First, note that the distribution of $t \mapsto X^{(k\gamma,k\gamma_0)}_{t/k}$ is independent of $k>0$. Second, the effective generator $\calL_{\eff}$ depends on the dominant generator only through the probability kernel $\Pa$. As a consequence, the effective process is formally invariant by a time-change of the dominant process, which amounts to multiply $\calL_{1}$ by a strictly positive function.
\end{Rem}

Our first technical assumption amounts to assume that $\Pa$ is continuous.
\begin{Ass}[Continuity of $\Pa$]\label{ass:P_cont}
The map $x \mapsto \Pa(x,\,.\,)$ is continuous on $E$ for the usual metric topology of convergence in distribution.
\end{Ass}
In particular, since for all $x_0\in E_\eff$,
$
\Pa(x_0,\,.\,)=\delta_{x_0},
$
, we obviously have  $\lim_{x \to x_0} \Pa(x,\,.\,)=\delta_{x_0}$ for all $x_0 \in E_\eff$. 

The next assumption is a natural compact containment condition required in order to obtain the tightness of the considered sequence of processes. This compact containment condition is not a necessary condition for tightness; it might be relaxed in specific cases. 
\begin{Ass}[Compact containment]\label{ass:tight} Let $\gamma_0 \in \set{0,1}$ be given. For any horizon time $T$, any $\eps > 0$ and any compact $K\subset E$, there exists a compact set $K_{\eps,T}\subset E$ such that:
\begin{equation}\label{eq:cont2}
\sup_{x\in K, \, \gamma > 0}\P_{x} \b{ \exists t\in [0,T], \,  X^{({\gamma},{\gamma_0})}_t \in K_{\eps,T} ^c } \leq \eps.
\end{equation}
\end{Ass}
Using a time scaling argument, it is possible to show the following direct consequence on the dominant process.
\begin{Lem} Assumption~\ref{ass:tight} implies that for any $\eps > 0$ and any compact $K\subset E$, there exists a compact set $K_{\eps,T}\subset E$ such that:
\begin{equation}\label{eq:cont_dom}
\sup_{x\in K}\P_{x} \b{ \exists t\in \R_+, \,  X^{(1,0)}_t \in K_{\eps} ^c } \leq \eps .
\end{equation}
\end{Lem}
\begin{proof}
    Let $T=1 $ be fixed. By a time rescaling argument and routine monotone convergence, it holds 
    \begin{align*}
        & \sup_{ \gamma > 0}\P_{x} \b{ \exists t\in [0,1], \,  X^{(\gamma,0)}_t \in K_{\eps} ^c }=\\
        & \quad \sup_{ \gamma > 0}\P_{x} \b{ \exists t\in [0,\gamma], \,  X^{(1,0)}_t \in K_{\eps} ^c }  = \P_{x} \b{ \exists t\in \R_+, \,  X^{(1,0)}_t \in K_{\eps} ^c }.
    \end{align*} 
    Then Assumption~\ref{ass:tight} yields the result.
\end{proof}
Note also that for a converging sequence of initial distributions $\mu_0^\gamma \to_{\gamma \to +\infty} \mu_0$ in $\Proba(E)$ (which form a tight family), the above assumption implies the existence of a compact $K_{\eps,T}$ such that 
\begin{equation}\label{eq:cont}
\sup_{\gamma > 0} \P_{\mu_0^\gamma} \b{ \exists t \in [0,T], \,  X^{(\gamma,\gamma_0)}_t \in K_{\eps,T} ^c } \leq \eps.
\end{equation}

We next present two additional assumptions that will be used to prove a (technical) result given in Proposition~\ref{cvunif} below that proves that the large time convergence of the dominant process towards $E_\eff$ is uniform with respect to initial conditions in a given compact set. The first one requires that the dominant process can be represented as stochastic flow (\textit{e.g.} as a strong solution of a Stochastic Differential equation parametrized by the initial condition) with some continuity with respect to the initial condition.
\begin{Ass}[Continuous stochastic flow]\label{ass:flow}The family of dominant processes indexed by their initial conditions can be constructed as a measurable map
$$
(\omega,t,x) \mapsto X_t^{(1,0)}(x)(\omega) \in E,
$$
from $\Omega \times [0,+\infty[ \times E$ to $E$ and such that, for each $t \in [0,+\infty[$, $x \mapsto X_t^{(1,0)}(x)$ is continuous in probability.
\end{Ass}
The second one is a necessary (see Remark~\ref{rem:assSEP} below for a counter-example) technical condition.
\begin{Ass}[Randomness of the hitting point]\label{ass:non_dirac}  The hitting point of the dominant process, $X^{(1,0)}_{\infty}(x)$, is deterministic if and only if $x\in E_{\eff}$. In the latter case ($x\in E_\eff$), we recall that $X^{(1,0)}_\infty(x)=x$.  
\end{Ass}

We can then obtain the following technical result (Proposition~\ref{cvunif}) that gives sufficient conditions implying uniform convergence in time on compacts of the dominant process.
\begin{Rem}[On the role of the above assumptions] Contrary to other assumptions, \emph{Assumption~\ref{ass:flow} and~\ref{ass:non_dirac} are only used for the proof of Proposition~\ref{cvunif}}. In the same way, Proposition~\ref{cvunif} is only required once in Step~$2$ (that states that the limit effective process lies in $E_\eff$) of the proof of the main homogenization theorem. 
\end{Rem}

\begin{Pro}[The large time convergence of the dominant process is uniform]\label{cvunif} Assume that the dominant process satisfies Assumption~\ref{ass:main}-\ref{ass:P_cont}-\ref{ass:tight}-\ref{ass:flow}-\ref{ass:non_dirac} (\textit{Nota Bene:} for the dominant process  Assumption~\ref{ass:tight} amounts to~\eqref{eq:cont_dom}). Then there exists a continuous bounded non-negative function $f$ with $f > 0$ on $E \setminus E_\eff$ and $f=0$ on $E_\eff$ such that for any compact $K \subset E $:
$$
\lim_{t \to + \infty} \sup_{x \in K}\E\b{f(X^{(1,0)}_t(x))} = 0 .
$$
\end{Pro}
The proof is this result is quite technical and is left at the end of Section~\ref{sec:proofs}. 
\begin{Rem}[Assumption~\ref{ass:non_dirac} is necessary]\label{rem:assSEP}
    The only unusual requirement of Proposition~\ref{cvunif} is Assumption~\ref{ass:non_dirac}. It is however unavoidable. Indeed, consider the deterministic process on the unit circle of generator $\sin^2(\frac{\theta}{2}) \cdot \partial_{\theta}$, with $\theta \in [0,2\pi[$. The process evolves counter-clockwise, and from any starting point different from $\theta = 0$ one eventually converges to $\theta = 0$ in infinite time. However, the hitting time of some $\theta_1 \neq 0$ starting with initial condition $\theta_0$ becomes infinite when $\theta_0 \to 0^+$: if we start just above $0$, we reach $\theta_1 > 0$ only after a very large time. As a consequence, the process cannot be arbitrarily close to $0$ at a large given time uniformly in the starting point; and the conclusion of Proposition~\ref{cvunif} cannot hold true.
\end{Rem}

The next continuity assumption is more benign, and ensures that the addition of a small perturbation process results in a small change of distribution.

\begin{Ass}[Continuous perturbation]\label{ass:contper} Let $\ph$ be any continuous bounded function. One has the continuity property: for all $t\geq 0$ and all $K \subset E$ compact,
$$
\lim_{\gamma \to + \infty} \sup_{x \in K}\abs{ \E(\ph(X^{(1,1/\gamma)}_t(x))) 
- \E(\ph(X^{(1,0)}_t(x)))} = 0
$$

\end{Ass}

We can finally consider assumptions related to the limit process and its generator. The first assumption is a technical requirement that ensures that a limiting effective generator indeed exists, with appropriate continuity properties.
\begin{Ass}[Existence of an effective generator]\label{ass:effgen} There exists an operator $\calL_{\mrm{eff}}$ acting on a vector space of measurable test functions of $E_\eff$ denoted $\calD_\eff$, and such that:

i) $\calD_{\eff}$ is a subset of continuous and bounded functions, contains constants, and contains a countable subset $(\varphi_i)_{i\in I}$ that separates points of $E_{\eff}$.

ii) To any $\ph \in \calD_\eff$, one can associate a bounded measurable function $\La_{(0)} \Pa \ph $ such that for any $ \gamma > 0$ and initial condition $x \in E$,
$$
t \mapsto \Pa \ph (X^{(\gamma,1)}_t(x))- \int_0^t \La_{(0)} \Pa \ph (X^{(\gamma,1)}_s(x)) \d s
$$
is a martingale for the natural filtration of $X^{(\gamma,1)}(x)$ (\textit{Nota Bene}: $(\gamma \calL_1 + \calL_0)\Pa = \calL_0 \Pa$, formally). 


iii) For all $\ph \in \calD_\eff$, $\calL_{0} \Pa \ph$ is continuous in $E$ at points of $E_{\eff}$. Explicitly: for any $\ph \in \calD_\eff$ and any $x_0 \in E_\eff$, it holds $\calL_{\mrm{eff}} \b{\ph}(x_0) = \lim_{x \to x_0} \calL_{0} \Pa \b{\ph}(x)$.
\end{Ass}
\begin{Rem}  Point ii) in Assumption~\ref{ass:effgen} amounts to computes the predictable plus martingale decomposition (Doob-Meyer decomposition) of $t \mapsto \Pa \ph (X_t^{(\gamma,1)})$ using It\^o formula or an appropriate variant, and to identify the predictable variant of $\La_{(0)} \Pa \ph (X_t^{(\gamma,1)}) \d t$. Typically, $\Pa \ph$ will belong to a 'natural' domain of the operator $\La_{(0)}$ (\textit{e.g.} differentiable if $\La_{(0)}$ is a first order differential operator on a manifold). Note that defining $\La_{(0)} \Pa \ph$ outside of $E\setminus E_{\eff}$ or even pointwise might not always be straightforward, and we may take the limit considered in Assumption \ref{ass:effgen} point $iii)$ as the new definition of $\La_{(0)} \Pa \ph$ on $E_{\eff}$. 
\end{Rem}

Finally, the following final assumption is necessary to identify the distribution of the effective Markov process on $E_\eff$, using the classical characterization of Markov processes by martingale problems (see \textit{e.g.}~\cite{EK}).
\begin{Ass}[Well-posedness of the effective martingale problem] \label{ass:wellposed} Let $\mu_0 = \lim_{\gamma \to +\infty} \mu_0^\gamma$ be given in $\Proba(E)$. The martingale problem $(\calL_{\mrm{eff}},\calD_\eff)$ for c\`adl\`ag processes with initial condition $\mu_0 \Pa $ is well-posed. 
\end{Ass}
Finally, note that it may be possible that in some cases, slightly different types of assumptions may be used:
\begin{Rem}[On the weakening of assumptions] It might happen that in some cases another method enables to obtain the uniform convergence result of Proposition~\ref{cvunif} (case i)), or, even more directly, another method enables to obtain Step~$2$ of the proof of the main homogenization theorem that states that the effective process must lie in $E_\eff$ (case ii)). We do not consider such examples in the present work, but several assumptions may then be weakened, and our proof can be revisited with less constraints.
\begin{itemize}
\item If case i) (or ii)) happens, Assumption~\ref{ass:main} can be weakened to convergence in distribution only. Assumption~\ref{ass:P_cont} can also be weakened to cases where $\Pa$ is only continuous at points of $E_\eff$. The proof is similar but one has to restrict in Theorem~\ref{main_theorem} to initial conditions $\mu_0^\gamma$ that are either constant (independent of $\gamma$) or that have a limit when $\gamma \to +\infty$ whose support lies in $E_\eff$.
\item If case i) (or ii)) happens, Assumption~\ref{ass:flow} and Assumption~\ref{ass:non_dirac} are non-longer required.
\item If case ii) happens, Assumption~\ref{ass:contper} is also no longer necessary. In case ii) the only remaining \emph{key points are the compact containment and the continuity of $\Pa \ph$ and $\calL_0 \Pa \ph$ at points of $E_\eff$}.
\end{itemize}
\end{Rem}

\subsection{The case of diffusions with Lipschitz coefficients}\label{sec:diff}

Let us now discuss more specific cases. In Section \ref{sec:examples}, we will present few examples that all belong to the following setting, and check that the assumptions are indeed verified.

Assume $E \subset \R^n$ is a closed domain with (piecewise) smooth boundary $\partial E$, and \emph{$E_{\eff} \subset E$ is a smooth sub-manifold of $E$}. In the examples of Section \ref{sec:examples}, $E_\eff$ is moreover a submanifold of $\partial E$. 

The process $X^{(\gamma_1,\gamma_0)}(x)$ is assumed to be a diffusion defined as the strong solution of a Stochastic Differential Equation (SDE) of the general form:
$$
\d X^{(\gamma_1,\gamma_0)}_t = \sqrt{\gamma_1} \sigma_1(X^{(\gamma_1,\gamma_0)}_t) \d W_t + \gamma_1 b_1(X^{(\gamma_1,\gamma_0)}_t) \d t + \gamma_0 b_0(X^{(\gamma_1,\gamma_0)}_t) \d t .
$$
with Lipschitz coefficients $(\sigma_1,b_1,b_0)$; $\calL_{0}={b_0} \cdot \nabla$ is thus a Lipschitz vector field. 

In that context, few (quite benign) assumptions can already be checked. Let us consider the diffusion process on the (slightly) extended state space $E \times [0,1]$ defined by:
$$
t \mapsto (X^{(1,E_t)}_t,E_t)
$$
in which $E_t=E_0=e$ is constant over time. The latter is also a strong solution of the SDE defined on $E \times [0,1]$ with Lipschitz coefficients, since the coefficient $(x,e) \mapsto (e \times b_0(x),0) $ is obviously Lipschitz. We know (see~\cite[Theorem~$13.1$]{williams1979diffusions}) that a strong solution of a SDE with Lipschitz coefficients can be represented by a stochastic map $(x,e,t) \mapsto X^{(1,e)}_t(x)$ which is almost surely continuous. This implies both Assumption~\ref{ass:flow} and Assumption~\ref{ass:contper} (by a routine dominated convergence argument and the Heine-Cantor theorem).

Note that it is necessary at this stage to check that solutions of the considered SDE indeed remain in $E$ for all time. Here is an example of a simple argument that enables to check it.
\begin{Lem} Let $d X_t = b(X_t) \d t + \sigma(X_t) \d W_t$ be a strong solution of a SDE with Lipschitz coefficients on $\R^n$. Assume that there is a smooth function $\xi:\R^n \to \R$ such that $E = \set{x: \, \xi(x) \leq 0}$ and verifying on $\R^n \setminus E$: i) $L \xi \eqdef b \cdot \nabla \xi +\frac12 \sigma \sigma^T : \nabla^2 \xi \leq 0$, as well as ii) $\nabla \xi \sigma =0$. Then $X_t$, starting from $E$ remains in $E$ for all time, that is $E$ is an absorbing set.
\end{Lem}
\begin{proof} Consider the semi-martingale $Z_t \eqdef \xi^+(X_t)$. By construction, the latter is predictable, so its local time at 0 vanishes, and the It\^o-Tanaka formula implies
$$
\d Z_t = \one_{Z_t >0} L\xi (X_t) \d t \leq 0.
$$
As a consequence, if $x\in E$ then $Z_t =0$ for all times, that is to say $X_t\in E$ for all times.
\end{proof}

In order to check that the dominant process do converge to a random point of $E_\eff$ (perhaps in infinite time, Assumption~\ref{ass:main}), and then to check Assumption~\ref{ass:P_cont}-\ref{ass:non_dirac} on the kernel $\Pa=\Law(X^{(1,0)}(x))$, one may consider a well-chosen time-change of the dominant process, again solution to a strong SDE, but which almost surely hits $E_\eff$ at finite time. This typically happens for instance when $\sigma_1\sigma_1^T$ is strictly positive in $\mathring E$ and vanishes only on $E_\eff$.

In that context, one will also usually consider compactly supported smooth functions
$
\calD_\eff = C_c^\infty(E_{\eff}),
$
which has a separable point separating subset given in a local chart by polynomials with rational coefficients (point i) of Assumption~\ref{ass:effgen}.

In order to prove ii) in Assumption~\ref{ass:effgen}; it is then sufficient to check that $x \mapsto  \Pa \ph (x)$ is, say, twice differentiable so that one can apply It\^o formula. One way to prove the regularity of $\Pa \ph$ is to use the regularity and uniqueness of solutions of the (\textit{e.g.} (hypo)elliptic or parabolic) Partial Differential Equation $\La_{1}\psi=0$ with smooth Dirichlet boundary condition $\psi=\ph$ at $E_\eff$ (we recall that $\Pa\varphi$ is a solution $\calL_1 \Pa \varphi = 0$ in an appropriate sense).

The (Doob-Meyer) finite variation part of $t \mapsto \Pa \ph (X_t^{(\gamma,1)})$ is eventually given by $\La_{(0)} \Pa \ph(X^{(\gamma,1)}_t) \d t$, $\La_{(0)} \Pa \ph$ being continuous on $E$. The value of $\La_{(0)} \Pa \ph$ on $E_{\eff}$ then defines the effective generator $\calL_{\mrm{eff}} \b{\ph}$, which is typically a Levy-type jump process.

Finally, at least when the effective dynamics can be described as a strong solution of a stochastic differential equation with Lipschitz coefficients and driven by, say, a Levy homogeneous process, some quite generic results can give the well-posedness of the martingale problem (see \textit{e.g.} the reference~\cite{kurtz2011}).

\subsection{Main homogenization theorem}
In order to state the main theorem, we first define the pseudo-path topology on path space used in strong noise homogenization problems.

\begin{Def}[Pseudo-paths] Let $E$ denotes a Polish space. To each measurable path $x: \R_+ \to E$, one can associate a probability on $E \times \R_+$ defined by:
\begin{equation}\label{eq:pseudo-def}
\delta_{x_t}(\d x) \e^{-t} \d t \in \Proba(E \times \R_+).
\end{equation}
The subset of probability distributions in $\Proba(E \times \R_+)$ of the form~\eqref{eq:pseudo-def} is called the pseudo-paths space. It is closed for the usual (Polish) topology of $\Proba(E \times \R_+)$ given by convergence in distribution. The pull-back on the equivalence classes of measurable paths equal Lebesgue almost everywhere will be called the pseudo-path topology. The associated space will be denoted $L^0(\e^{-t} \d t, E)$.
\end{Def}
\begin{Rem}[Basic properties]
The fact that the space of pseudo-paths is closed for convergence in distribution and other main results related to this topology are summarized in Section~\ref{sec:MeyerZheng}. In particular, let us recall that pseudo-paths topology inherits the usual sequential characterization of convergence in distribution: a sequence of paths converges for the pseudo-paths topology if and only if $\lim_n \int \ph(x^n_t,t) \e^{-t} \d t = \int \ph(x^\infty_t,t) \e^{-t} \d t  $ for all continuous and bounded $\ph$. This convergence is in fact equivalent to the stronger convergence ``in probability'' (also called ``in measure'') defined by $\lim_n \int_0^\infty \min(1,d(x^n_t,x_t)) \e^{-t} \d t$ where $d$ metrizes $E$. For that reason, the pseudo path space is denoted $L^0(\e^{-t} \d t, E)$.
\end{Rem}

We can now state our main theorem:
\begin{The}[Strong noise homogenization] \label{main_theorem}
Consider $\La_{(1)}$ and $\La_{(0)}$ two Markov generators, and $(\mu_0^\gamma)_{\gamma > 0}$ a family of initial distributions with support in $E$ converging in distribution towards $\mu_0$ a distribution on $E$. Let Assumptions \ref{ass:main}-\ref{ass:P_cont}-\ref{ass:tight}-\ref{ass:flow}-\ref{ass:non_dirac}-\ref{ass:contper}-\ref{ass:effgen}-\ref{ass:wellposed} be fulfilled. Then a c\`adl\`ag process $X^{(\gamma,1)}$ on $E$ solution of the martingale problem of generator $\La_{\gamma}=\gamma \La_{(1)}+\La_{(0)}$ with initial condition $\mu_0^\gamma$ converges in distribution when $\gamma \to + \infty$ for the (Meyer-Zheng) pseudo-path topology to the unique c\`adl\`ag Markov process $X^\infty$ on $E_{\eff}$ solution of the martingale problem of generator $\La_{\eff}$ and initial condition $\mu_0 \Pa$.
\end{The}

The convergence in law of $X^{\gamma}$ towards $X^\infty$ for the pseudo-path topology implies by general considerations of the pseudo-path topology the convergence in law of $(X_{t_1}, \ldots, X_{t_k}) \in E^k$ for Lebesgue almost all time sequences $(t_1, \ldots ,t_k)$ in $\R_+^k$. We can in fact prove that the convergence is actually true for all time sequences: finite-dimensional distributions do converge.

\begin{Cor}[Convergence of finite-dimensional distributions] \label{conv_simple}
Consider the same notations as in Theorem \ref{main_theorem}. Under the same assumptions, and for all $(t_1, \ldots,t_k)\in \R_+^k$, $k \geq 1$, the random sequence $(X_{t_1}^{(\gamma,1)},\ldots,X_{t_k}^{(\gamma,1)})$ converges in law to $(X^\infty_{t_1}, \ldots,X^\infty_{t_k})$.
\end{Cor}

\section{Proof of the main theorems}\label{sec:proofs}

In order to prove Theorem~\ref{main_theorem}, we will proceed in four steps. First, using a tightness argument, we will extract a sub-sequence of the considered class of processes which is converging for the pseudo-path topology. We denote the extracted limit $X^\infty$, and we underline that $X^\infty$ is really a pseudo-path, that is the random distribution $\delta_{X_t}( \d x) \e^{-t} \d t$. Next, we will prove that $X^\infty$ takes its values in $E_{\eff}$, which amounts to say that almost surely $X_t^{\infty}\in E_{\eff}$ for Lebesgue almost all $t>0$. Then, using the so-called Meyer-Zheng criteria (see Section~\ref{sec:MeyerZheng}), we will prove that $X^{\infty}$ can be identified with a c\`adl\`ag process, in the sense that the distribution of $X^\infty$ is the push-forward on pseudo-path space of the distribution of a c\`adl\`ag random process. Eventually, we will prove that $X^{\infty}$ is a solution of a martingale problem on $E_{\eff}$ (which have a unique solution by assumption). We conclude that $X^{(\gamma,1)}$ converges in law to $X^{\infty}$ for the pseudo-path topology. We stress that the third step really is unavoidable: limits with respect to the pseudo-path topology,  \emph{without the Meyer-Zheng criteria}, may be too degenerate to inherit the martingale property.

In the penultimate section we will prove that we have furthermore convergence of all finite-dimensional distributions, and in the last section, we will prove Proposition~\ref{cvunif}.

\subsection{{Step~$1$:} Tightness and extraction}

We start by proving that the considered family of processes is tight when considering the pseudo-path topology. To do so, one starts with the following almost trivial remark:
\begin{Lem}\label{lem:comp} Let $(K_n)_{n \geq 0}$ denotes any given increasing sequence of compact subsets of $E$. The set of pseudo-paths defined by:
\[
\calK_{(K)} \eqdef \bigcap_{n \geq 0} \set{x: \, \int_0^{n} \one_{x_t \in K^c_n} \e^{-t} \d t = 0 }
\]
is relatively compact for the pseudo-path topology.
\end{Lem}
\begin{proof} Let us consider $\calK_{(K)}$ as a subset of the space $\Proba(E \times \R_+)$, that is we identify $x \in \calK_{(K)}$ with probability $ \delta_{x_t}(\d x) \e^{-t} \d t$. Let $\eps > 0$ be given, arbitrarily small. Let $n_\eps$ be such that $\e^{-n_\eps} \leq \eps$. One has by construction of $\calK_{(K)}$ that for any $x \in \calK_{(K)}$
$$
\int_{E}\int_{0}^{\infty} \one_{ (x,t) \in \p{K_{n_\eps} \times [0,n_\eps]}^c} \delta_{x_t}(\d x) \e^{-t} \d t \leq \e^{-n_\eps} \leq \eps.
$$
By Prokhorov theorem, $\calK_{(K)}$ is thus relatively compact.
\end{proof}

We can now proceed and prove that under Assumption~\ref{ass:tight}, the considered process is tight under the pseudo-path topology.
\begin{Lem}Under Assumption~\ref{ass:tight}, the family $\p{X^{(\gamma,1)}}_{\gamma > 0}$ is tight under the pseudo-path topology.
\end{Lem}
\begin{proof} Let $\eps > 0$ be given, arbitrarily small. By Lemma~\ref{lem:comp}, it sufficient in order to conclude to prove that there exists a increasing sequence $K_{\eps,n} \subset E$ of compact sets such that
$$
\P\b{ X^{(\gamma,1)} \in \calK_{(K_{\eps})}^c } \leq \eps .
$$

For this purpose, using Assumption \ref{ass:tight}, we choose $K_{\eps,n}$ such that we have $ \P( \exists t \in [0,n], \, X^{(\gamma,1)}_t \in K_{\eps,n}^c ) \leq \eps 2^{-n}.$

Then by construction:
$$
\P \b{ X^{(\gamma,1)} \in \calK_{(K_\eps)}^c } \leq \sum_n \P\b{\exists t \in [0,n], \, X^{(\gamma,1)}_t \in K_{\eps,n}^c } \leq \eps.
$$

\end{proof}

As a consequence, any increasing sequence converging to infinity has an increasing sub-sequence denoted $(\gamma_p)_{p \geq 0}$ such that $X^{(\gamma_p,1)}$, when embedded as a pseudo-path in $\Proba(E\times \R_+)$, converges when $p$ goes to infinity in distribution towards a distribution in $\Proba(E\times \R_+)$. Since we know that the set of pseudo-paths is closed, the limit is a pseudo-path and can be represented as a random process denoted $X^{\infty}$. However, $X^{\infty}$ may not have a c\`adl\`ag representative at this stage (this will be taken care of in Step~$3$ below). 

\subsection{{Step~$2$:} The limit process lies in $E_{\eff}$}

The goal of this step is to prove that the considered limit process takes its values in $E_{\eff}$ in the sense that:
$$
\E\b{ \int_0^{+\infty} \one_{X^\infty_t \in E_{\eff}} \e^{-t}\d t  } = 1 .
$$

We now use Proposition \ref{cvunif} to prove that the full process approaches $E_{\eff}$ when $\gamma$ approaches infinity:
\begin{Lem}\label{lem:bords} Let $f$ denotes a continuous non-negative function on $E$ vanishing on $E_{\eff}$ and strictly positive on $E \setminus E_{\eff}$ characterizing the uniform convergence of the dominant process towards $E_{\eff}$ in Proposition~\ref{cvunif}. Let $\mu_0^\gamma$ denotes the initial condition of $X_t^{(\gamma,1)}$. Under Assumptions \ref{ass:tight} and \ref{ass:contper}, we have for all $t>0$ convergence towards $E_{\eff}$ for large $\gamma$:
$$\E_{\mu_0^\gamma}(f(X_t^{(\gamma,1)}))\underset{\gamma\to+\infty}{\longrightarrow} 0.$$
\end{Lem}
\begin{proof}
First notice that
\begin{align*}\E_{\mu_0^\gamma}(f(X_t^{(\gamma,1)}))&=\langle \mu_0^\gamma, e^{t(\gamma \La_{(1)}+\La_{(0)})}f\rangle \\ &=\langle \mu_0^\gamma, e^{t\gamma( \La_{(1)}+\frac{1}{\gamma}\La_{(0)})}f\rangle \\ &=\E_{\mu_0^\gamma}(f(X_{\gamma\times t}^{(1,\frac{1}{\gamma})})),\end{align*}
which amounts to an elementary time change.

Therefore, for any  $0<\delta<\gamma\times t$ and any compact subset $K$ of $E$ we have:

\begin{align*}
&\abs{\E_{\mu_0^\gamma}\left(f\left(X_{\gamma\times t}^{(1,\frac{1}{\gamma})}\right)\right)}\\
&\leq \abs{\E_{\mu_0^\gamma}\left(f\left(X_{\gamma\times t}^{(1,\frac{1}{\gamma})}\right)\right) 
- \E_{\mu_0^\gamma}\left(f\left(X_{\gamma \times t}^{(1,0)}\right)\right)}+\abs{\E_{\mu_0^\gamma}\left(f\left(X_{\gamma\times t}^{(1,0)}\right)\right)}\\
&=\abs{\E_{\mu_0^\gamma}\b{ \E_{X_{\gamma\times t-\delta}^{(1,\frac{1}{\gamma})}}\left(f\left(X_{\delta}^{(1,\frac{1}{\gamma})}\right)\right) 
- \E_{X_{\gamma\times t-\delta}^{(1,0)}}\left(f\left(X_{\delta}^{(1,0)}\right)\right)  }} \\
&\ \ \ \ \ \  +\abs{\E_{\mu_0^\gamma}\left(f\left(X_{\gamma\times t}^{(1,0)}\right)\right)}\\
&\le\abs{\E_{\mu_0^\gamma}\b{\E_{X_{\gamma\times t-\delta}^{(1,\frac{1}{\gamma})}}\left(f\left(X_{\delta}^{(1,\frac{1}{\gamma})}\right)\right) 
- \E_{X_{\gamma\times t-\delta}^{(1,\frac{1}{\gamma})}}\left(f\left(X_{\delta}^{(1,0)}\right)\right)}} \\
& \ \ \ \ \ \ +\abs{\E_{\mu_0^\gamma}\b{\E_{X_{\gamma\times t-\delta}^{(1,\frac{1}{\gamma})}}\left(f\left(X_{\delta}^{(1,0)}\right)\right) 
- \E_{X_{\gamma\times t-\delta}^{(1,0)}}\left(f\left(X_{\delta}^{(1,0)}\right)\right)}} \\
& \ \ \ \ \ \  \ \ \ \ \ \ +\abs{\E_{\mu_0^\gamma}\left(f\left(X_{\gamma\times t}^{(1,0)}\right)\right)} \\
&\leq \sup_{x \in K }\abs{ \E(f(X^{(1,\frac{1}{\gamma})}_\delta(x))) 
- \E(f(X^{(1,0)}_\delta(x)))} + 2 \norm{f}_\infty \P_{\mu_0^\gamma}\b{X_{ t-\delta/\gamma}^{(\gamma,1)} \in K^c} \\
& \ \ \ \ \  +2\times \sup_{x \in K}\E\b{f(X^{(1,0)}_{\delta}(x))} + \norm{f}_\infty \P_{\mu_0^\gamma}\b{X_{ t-\delta/\gamma}^{(\gamma,1)} \in K^c}   \\
& \ \ \ \ \ \  \ \ \   + \norm{f}_\infty \P_{\mu_0^\gamma}\b{X_{ t-\delta/\gamma}^{(\gamma,0)} \in K^c}+\sup_{x \in K}\abs{\E\left(f\left(X_{\gamma\times t}^{(1,0)}(x)\right)\right)} +\|f\|_{\infty}\, \mu_0^\gamma(K^c) \numberthis{} \label{inegalite}
\end{align*}

Let $\eps >0$ be given, arbitrarily small. Using Assumption~\ref{ass:tight} (compact containment), one can choose $K=K_\eps$ such that, uniformly in $\delta,\gamma$:
$$ \max\p{ \P_{\mu_0^\gamma}\b{X_{ t-\delta/\gamma}^{(\gamma,1)} \in K_\eps^c },\P_{\mu_0^\gamma}\b{X_{ t-\delta/\gamma}^{(\gamma,0)} \in K_\eps^c}, \mu_0^\gamma(K_\eps^c)} \leq \frac{\eps}{10\|f\|_{\infty}}.
$$
Using Proposition \ref{cvunif} (uniform time convergence of the dominant process), we consider now $\delta_\eps \equiv \delta_{\eps,K_\eps}$ big enough such that 
$$\sup_{x \in K_\eps}\E\b{f(X^{(1,0)}_{\delta_\eps}(x))}\leq \frac{\eps}{10}$$
Now using Assumption \ref{ass:contper} (uniformly continuous perturbation), we now consider $\gamma_\eps \equiv \gamma_{\eps,K_\eps} >0$ such that for all $\gamma > \gamma_\eps$: 
$$\sup_{x \in K_\eps}\abs{ \E(f(X_{\delta_\eps}^{(1,1/\gamma)}(x))) 
- \E(f(X_{\delta_\eps}^{(1,0)}(x)))}\leq \frac{\eps}{5}$$
Taking $\delta= \delta_\eps$ inequality \eqref{inegalite} thus implies for all $\gamma>\max\left(\frac{\delta_\eps}{t}, \gamma_\eps\right)$: 
\begin{equation*}\abs{\E_{\mu_0^\gamma}\left(f\left(X_{t}^{(\gamma,1)}\right)\right)}\leq \eps .\end{equation*}
\end{proof}

We have proved in Step~$1$ that there exists $(\gamma_n)_n$ an increasing sequence of positive numbers approaching infinity such that we have $X^{(\gamma_n,1)}\overset{\La}{\longrightarrow} X^{\infty}$ for the pseudo-path topology. By characterization of the pseudo-path convergence, it implies that
$$\mathbb{E}\int_0^{+\infty} f\left(X_t^{(\gamma_n,1)}\right)\e^{-t} \d t \underset{n\to+\infty}{\longrightarrow}\mathbb{E}\int_0^{+\infty} f(X_t^{\infty}) \e^{-t} \d t.
$$
By dominated convergence and Lemma \ref{lem:bords} above, the term on the left converges to $0$. But $f^{-1}(\{0\})=E_{\eff}$ and $f\geq 0$, so eventually $X^{\infty}\in E_{\eff}$ almost surely as a pseudo-path (that is for Lebesgue almost all $t$).

\subsection{{Step~$3$:} The limit is c\`adl\`ag}

We will prove that the distribution of $X^{\infty}$ has a c\`adl\`ag representative taking values in $E_{\eff}$. To do so, we will use the so-called \emph{Meyer-Zheng criteria}. The latter ensures that a sequence of \emph{real valued} random processes that i) converge in distribution for the pseudo-path topology, and ii) has uniformly bounded \emph{mean variations} (see below for a definition), has a pseudo-path limit which possesses a c\`adl\`ag representative.

\begin{Def}[Mean variation] Let $T > 0$ denotes an horizon time, and $(Z_t)_{t \geq 0}$ a c\`adl\`ag real valued random process with $Z_t \in L^1(\P)$ for all $t$. The mean variation of $Z$ over $[0,T]$ is given by
\begin{align*}
V_T(Z) \eqdef & \sup_{0=t_0\leq t_1 \ldots \leq t_K \leq T} \sum_{k=0}^K \E \abs{\E\b{ Z_{t_{k+1}}-Z_{t_k} \mid \sigma\p{ Z_{t}, \, t \leq t_k }} }, \\
= & \underset{\abs{H}\leq 1}{\sup}\mathbb{E} \left( \int_0^{T} H_t \, \d Z_t \right),
\end{align*}
where in the above the supremum is taken over predictable processes (with respect to the natural filtration of $Z$) taking value in $[-1,1]$.
\end{Def}
We recall in Section~\ref{sec:MeyerZheng} several general facts underlying the Meyer-Zheng approach to limit theorems for processes. In particular, processes with bounded mean variations are quasi-martingales, and their bounded mean variation is precisely equal to the total variation of their predictable finite-variation part (in the sense of the Doob-Meyer decomposition). 

Step~$3$ (the limit process $X^\infty$ has a -- unique -- c\`adl\`ag representative) will be a direct consequence of Lemma~\ref{lem:cadlag}. We start with a technical remark that will be used in the proof of Lemma~\ref{lem:cadlag}.

\begin{Lem}\label{lem:measurab} Let $F$ denotes a Polish space, and let $\Psi: F \to \R^I$, with $I$ countable, be an injective bounded continuous function. Consider the weaker topology on $F$ induced by the pull-back by $\Psi$ of the product (Polish) topology of $\R^I$, for instance as defined by the distance
\begin{equation}\label{eq:dist_new}
\widetilde{d}(x,y) = \sum_{i \in I} 2^{-n_i}  \abs{\Psi_i(x) - \Psi_i(y) },
\end{equation}
where $i \mapsto n_i \in \N$ is injective. Denote by $\widetilde F$ the completion of $F$ with respect to the latter distance. Then i) $\Psi$ defines a homeomorphism on $\widetilde F$ for the weaker topology defined by $\widetilde d$, and ii) all Borel sets of $F$ defined by the strong original topology (defined by $d$) are also Borel in $\widetilde F$ for the weaker new topology (defined by $\widetilde d$).
\end{Lem}
\begin{proof} Point i) is trivial: $\widetilde{d}$ is a metric because the family $(\Psi_i)_{i \in I}$ is separating. Then $\widetilde F$ is simply identified with $\overline{\Psi(F)}$ by the isometry defined by $\widetilde d$.

Point ii) is more subtle. A standard result in descriptive set theory (\cite[Theorem~$15.1$]{kechris2012classical}) states that on Polish spaces, injective Borel maps are in fact Borel isomorphisms, in the sense that they map Borel sets to Borel sets. We know that $\Psi: (F,d) \to \R^I$ is continuous, it is thus Borel and so is its inverse $\Psi^{-1}$ defined on Polish $\overline{\Psi(F)}$. Let $B \subset F$ be Borel for the strong topology $\Phi(B)$ is Borel in $\overline{\Psi(F)}$, and $B \subset F$ is eventually Borel in Polish $(\tilde{F},\tilde d)$.
\end{proof}

\begin{Lem} \label{lem:cadlag} Let $X^n$ be a sequence of c\`adl\`ag processes on $E$ converging in distribution to a pseudo-path $X^{\infty}$ for the pseudo-path topology, and let $E_{\eff}\subset E$. We assume the following:
\begin{enumerate}
    \item The sequence $X^n$ satisfies the compact containment condition~\eqref{eq:cont} (the index $n$ playing the role of $\gamma$).
    \item The limit pseudo-path $X^\infty$ almost surely has support in $E_\eff$.
    \item There exists a countable family of bounded continuous map $\Psi_i: E\to\R$, $i\in I$ such that for all $T \geq 0$ and all $i \in I$:
    $$
    \sup_n V_T( \mathrm{Law}(\Psi_i(X^n)) < + \infty .
    $$
    In other words the (quasi-martingale) c\`adl\`ag processes $\Psi_i(X^n)$ have a uniformly (w.r.t. $n$) bounded mean variation on finite time intervals (Meyer--Zheng criteria). 
    \item The family $(\Psi_i)_{i \in I}$, separates the points of $E_{\eff}$.
\end{enumerate}
Then 
the distribution of $X^\infty$ is given by the distribution of a c\`adl\`ag process taking its values in $E_{\eff}$.
\end{Lem}
\begin{proof}
First remark that Point~$2$ in the lemma's assumption means that any process representative of $X^\infty$ is such that $\e^{-t}\d t \otimes \P(\d \omega)$-almost everywhere we have $X^\infty_t(\omega) \in E_\eff$. Equivalently, $X^\infty$ as a random pseudo-path almost surely belongs to the pseudo-path subspace $L^0(\e^{-t} \d t, E_\eff) \subset L^0(\e^{-t} \d t, E)$. 

On the other hand, by continuity, the distribution of the bounded random pseudo-path $\Psi_i(X^n) \in \R$ converges when $n\to +\infty$ towards the distribution of the random pseudo-path $\Psi_i(X^\infty)$. For each $i \in I$, Point~$3$ in the assumption is exactly the uniform mean variation condition of Meyer-Zheng theory (see Theorem~\ref{the:cadlagMZ} in Section~\ref{sec:MeyerZheng}), which ensures that the distribution of the pseudo-path $\Psi_i(X^\infty)$ is in fact supported by $\D(\R_+,\R)$ (the space of c\`adl\`ag paths on $\R$). This implies that the following canonical c\`adl\`ag representative of $\Psi_i(X^\infty)$ 
$$
Z^i_t = \lim_{h \to 0} \frac1h \int_t^{t+h} \Psi_i(X^\infty_s) \d s
$$
does exist almost surely, and is by construction a real-valued c\`adl\`ag process. The set of elements $(\omega,t) \in \Omega \times \R_+$ such that $\Psi_i(X^\infty_t(\omega))=Z^i_t(\omega)$ has $\e^{-t}\d t \otimes \P$-measure $1$, and since $I$ is countable, the intersection for $i \in I$ is again of $\e^{-t}\d t \otimes \P$-measure $1$.

  Using Point~$4$ ($\Psi$ is injective on $E_\eff$) and Lemma~\ref{lem:measurab}, one can then consider the extension of the inverse $\Psi^{-1}: \overline{\Psi(E_\eff)} \to \widetilde{E}_\eff$ which defines an isometry when one considers \textit{e.g.} the distance~\eqref{eq:dist_new} associated with the product Polish topology of $\R^I$. With respect to this weaker topology, the process $t \mapsto \widetilde{X}^\infty \eqdef \Psi^{-1}(Z_t) \in \widetilde{E}_\eff$ is almost surely c\`adl\`ag. By Lemma~\ref{lem:measurab} point ii), Borel sets of $E_\eff$ for weaker topology are also Borel for the initial stronger topology, so that if $\ph$ is a continuous bounded function for the strong topology, the integral $\int_0^{+\infty}\ph(\widetilde{X}^\infty_t,t) \e^{-t} \d t $ is well defined and it holds by construction that
  $$
  \int_0^{+\infty}\ph(\widetilde{X}^\infty_t,t) \e^{-t} \d t = \int_0^{+\infty}\ph(X^\infty_t,t) \e^{-t} \d t .
  $$
  $\widetilde{X}^\infty$ is thus a measurable representative of the random pseudo-path $X^\infty$.

It remains, in order to conclude the proof, to show that $\widetilde{X}^\infty$ takes its values in $E_\eff$ and is c\`adl\`ag for the stronger original topology of $E$. In order to do so, we are to use a compactness argument (Point~$1$) and show that $\widetilde X^\infty$ has, almost surely, left and right accumulation points in the original stronger topology of $E$. This is not guaranteed in general without the compact containment condition because the weaker topology of $\widetilde{E}_\eff$ may allow values in $\widetilde{E}_\eff \setminus  E_\eff$.

 More precisely, we now claim that with the compact containment condition (Point $1$), for any $T >0$, the set $\set{\widetilde X^\infty_t(\omega), \, t \in [0,T]}$ is $\P(\d \omega)$-almost surely a subset of $E$ that is relatively compact for the stronger original topology. If this claim is true, $\P(\d \omega)$-almost surely, for any (left or right) converging sequence $t_m \to t_\infty$, there is a sub-sequence (we do not change notation for the sub-sequence) such that $\widetilde X^\infty_{t_m}(\omega)$ converges in $E$ and in fact in $E_{\eff}$ since the set is closed. But $t \mapsto \widetilde X^\infty_t(\omega)$ is c\`adl\`ag for the weaker topology of $\widetilde{E}_\eff$, so the limit is unique and given by $\widetilde X^\infty_{t_\infty^{\pm}}(\omega)$. Hence   $\widetilde X^\infty$ is c\`adl\`ag in $E_\eff$ and the proof of the whole lemma is complete.

It remains to prove the claim above. We consider the original topology on $E_\eff$. We first remark that by the portmanteau theorem, the map on pseudo-paths $x \mapsto \int_0^T \one_{x_t \in K^c} \e^{-t} \d t$ is lower semi-continuous for $K$ closed, so that $\set{ x: \, \int_0^T \one_{x_t \in K^c} \e^{-t} \d t }$ is open and applying the portmanteau theorem again albeit in pseudo-path space yields
$$
\P\b{ \int_0^T \one_{ X^\infty_t \in K^c} \e^{-t} \d t > 0 } \leq \liminf_n \P\b{ \int_0^T \one_{ X^n_t \in K^c} \e^{-t} \d t > 0 }.
$$
The compact containment condition implies that the right-hand side in the above is smaller than any $\eps$ for a well chosen $K \equiv K_\eps$. Using a routine Borel-Cantelli argument obtained by considering the events constructed for $K_{2^{-p}}$, $p \in \N$, one gets that:
$$
\P\b{ \exists p \geq 0: \,  \int_0^T \one_{X^\infty_t \in K_{2^{-p}}^c} \e^{-t} \d t = 0 } = 1 .
$$
Since $X^{\infty}$ is a pseudo-path of $E_{\eff}$, we can here replace $K_{2^{-p}}$ by $K_{2^{-p}}\cap E_{\eff}$ that is still a compact since $E_{\eff}$ closed. The above statement holds for $X^\infty$ taking value in pseudo-paths space, so that one can replace $X^\infty$ with $\widetilde X^\infty$. In the weaker topology of $\widetilde E_\eff$, the compact $K_{2^{-k}}\cap E_{\eff} \subset \widetilde E_\eff$ is again compact. Since $\widetilde X^\infty$ is c\`adl\`ag, it implies that $\P(\d \omega)$-almost surely $\widetilde X^\infty_t(\omega) \in K_{P(\omega)}$ for $t\in [0,T]$ and a random integer $P(\omega)$. This precisely means that the set $\set{\widetilde X^\infty_t(\omega), \, t \in [0,T]}$ is relatively compact for the strong original topology of $E_\eff$.  Our claim is proved, and so is the lemma.

\end{proof}

Let $\varphi_i$ denotes a countable family in $\calD_{\eff}$ that separates points. We now consider $$X^n=X^{(\gamma_n,1)}, \qquad \Psi_i=\Pa \varphi_i,$$
and try to apply Lemma~\ref{lem:cadlag}. We have already proved in Step~$2$ the second condition while the first one is Assumption~\ref{ass:tight} and the fourth one is part i) of Assumption~\ref{ass:effgen}. For the third one, we write:
$$V_{T}(\Psi_i(X^{n}))=\underset{\abs{H}\leq 1}{\sup}\int_0^{T}\mathbb{E}\left( H_t \, \d \Psi_i(X_t^{n}) \right),$$
where the supremum is taken over predictable processes (with respect to the natural filtration of $X^n$) taking value in $[-1,1]$. Since the martingale term disappear:
\begin{align*}V_{T}(\Psi_i(X^{n}))&=\underset{\abs{H}\leq 1}{\sup}\int_0^{T}\mathbb{E}(H_t \, \La^{(\gamma_n,1)} \Pa \ph_i(X_t^{n}) )\mathrm{d}t\\
&=\underset{\abs{H}\leq 1}{\sup}\int_0^{T}\mathbb{E}(H_t \, \La_{0} \Pa \ph_i(X_t^{n}) )\mathrm{d}t\\
& = \int_{0}^{T}\mathbb{E}(\abs{\La_{(0)}\Pa \ph_i(X_t^{n})})\mathrm{d}t \\
& \leq T \norm{\La_{(0)}\Pa \ph_i}_\infty,
\end{align*}
that is finite using part ii) of Assumption \ref{ass:effgen}. This yields the uniform mean variation condition. We stress that the above estimate is the key ingredient of the strong homogenization theorem.

Lemma \ref{lem:cadlag} gives us therefore that $X^{\infty}$ is almost surely c\`adl\`ag with trajectories in $E_{\eff}$.

\subsection{{Step~$4$}: The limit is the unique solution of a martingale problem}

We can now prove that $X^{\infty}$ is solution of a specific martingale problem. We already know by Assumption~\ref{ass:effgen} that for all $\gamma>0$ and $\varphi \in \calD_{\eff}$, $X^{\gamma}=X^{(\gamma,1)}$ is solution of: $$ \calM^{\gamma}:t\mapsto\Pa \varphi(X_t^{\gamma})-\int_0^t \calL_{0} \Pa \varphi (X_s^{\gamma})\mathrm{d}s \ \ \  \text{is a $\sigma(X^\gamma_s,s \leq t)$-martingale.}$$

We would like to let $\gamma$ goes to infinity in this problem, and get that:
$$\calM^{\infty}:t\mapsto \Pa\varphi(X_t^{\infty})-\int_0^t \calL_{0} \Pa \varphi(X_s^{\infty})\mathrm{d}s \ \ \  \text{is a $\sigma(X^\infty_s,s \leq t)$-martingale}.$$
To prove this, we will first prove using a general result on convergence in distribution for the pseudo-paths topology that Lebesgue almost all finite-dimensional distributions of $\calM^{\gamma_k}$ converges to those of $\calM^{\infty}$. In particular this will enable to identify the distribution of $X^\infty_0$. We will then invoke a routine lemma characterizing c\`adl\`ag martingales (Theorem~\ref{martingale_criteria}) in order to check that $\calM^{\infty}$ is indeed a martingale. The martingale problem uniqueness in Assumption~\ref{ass:wellposed} enables to identify the distribution of $X^\infty$ and then to conclude.

Using the general Lemma~\ref{lem:finiteMZ} for converging distributions for the pseudo-path topology, we know that there exists a subset with full Lebesgue measure $J \subset \R_+$ and a sub-sequence $\gamma_k \to +\infty$ such that for each $t_1, \ldots , t_p$ the joint variable
\[
(X^{\gamma_p},X^{\gamma_k}_{t_1}, \ldots,X^{\gamma_k}_{t_p}) \xrightarrow[p \to +\infty]{\Law} (X^{\infty},X^{\infty}_{t_1}, \ldots,X^{\infty}_{t_p})
\]
converges in distribution in $L^0(\e^{-t} \d t ,E)\times E^p$. Since by Assumption~\ref{ass:effgen} point iii), $\calL_{0} \Pa \varphi$ is bounded and continuous in $E$ at each point of $E_{\eff}$, the map
\[
x \mapsto \int_0^t \calL_{0} \Pa \varphi (x_s) \d s
\]
is bounded and continuous at points of $L^0(e^{-t}\mathrm{d}t,E_{\eff})$ for the pseudo-path topology.
In the same way, Assumption~\ref{ass:P_cont} ensures that the map $x \mapsto \Pa\varphi(x)$ is bounded and continuous. Combining the above results, we get that Lebesgue almost all (for times taken in $J$) finite-dimensional distributions of $\calM^{\gamma_p}$ do converges to those of $\calM^{\infty}$.

We now claim that this proves that:
$$
\Law(X^\infty_{0^+}) = \Pa \mu_0.
$$
First recall that the trajectories of $X^{\infty}$ are almost surely in $E_{\eff}$, and we have almost surely that $\Pa \varphi(X^{\infty})=\varphi(X^{\infty})$.
On the other hand by construction $\E\b{\calM^{\gamma}_t}=\E(\Pa \ph (X^\gamma_0))=\mu_0^{\gamma} \Pa \ph$ for each $t$, so that passing to the limit for $t \in \J$
$$
\E\b{\calM^{\infty}_t} = \mu_0 \Pa \ph
$$
by Assumption~\ref{ass:P_cont}, and since $\calM^{\infty}_{0^+} = \ph(X^\infty_{0^+})$ the claim is proved.

We now turn to the martingale property. Using the convergence of finite-dimensional distributions on $J$, one gets: for all $0\leq t_1 < ... < t_p < t$ in $J$, for all $\varphi_1, ... , \varphi_p$ continuous bounded, and for all $p$:
$$\xymatrix{
    \mathbb{E}[ (\calM_t^{\gamma}-\calM_{t_p}^{\gamma})\varphi_p(X_{t_p}^{\gamma}) ... \varphi_1(X_{t_1}^{\gamma}) ] = \ar[d]  & 0 \ar[d] \\
    \mathbb{E}[ (\calM_t^{\infty}-\calM_{t_p}^{\infty})\varphi_p(X_{t_p}^{\infty}) ... \varphi_1(X_{t_1}^{\infty}) ] = & 0 .
  }$$ 
Since the processes $\calM_{\cdot\wedge T}^{\infty}$ is a c\`adl\`ag process, we have $\calM^{\infty}$ is a martingale for the filtration of $X^\infty$ by Theorem \ref{martingale_criteria}.

Finally, recalling again that $\Pa \varphi(X^{\infty})=\varphi(X^{\infty})$ (trajectories of $X^{\infty}$ are almost surely in $E_{\eff}$), we have that:
$$t\mapsto \varphi(X_t^{\infty})-\int_0^t \calL_{0} \Pa \varphi (X_s^{\infty}) \ \ \ \text{is a $\sigma(X^\infty_s,s \leq t)$-martingale.}$$
Assumption~\ref{ass:wellposed} ensures eventually that the distribution of $X^{\infty}$ is the unique càdlàg solution of the above martingale problem.

Routinely combining this identification of the limit with the tightness of $\left(X^{\gamma}\right)_{\gamma>0}$ in pseudo-path space obtained in Step~$1$, we find that $X^{\gamma}$ indeed converges in law for the pseudo-path topology to $X^\infty$, the unique (in law) process on $E_{\eff}$ of generator $\La_{\eff}$: the Theorem \ref{main_theorem} is proved.

\subsection{Proof of Corollary \ref{conv_simple}}

In order to prove the convergence of the finite-dimensional laws for all times, we first need to prove that we have convergence of the semi-group for all times starting from a moving initial condition.
\begin{Pro}\label{undim}
    We assume as before that $X^{\gamma}$ starts from the initial condition $\mu_0^{\gamma}$ that converges in law in distribution towards $\mu_0$. Then for all $t>0$ we have that $X_t^{\gamma}$ converges in law towards $X_t^{\infty}$ of initial condition $\mu_0\Pa$.
\end{Pro}
\begin{proof} As in the proof of Step~$4$ above, since $X^{\gamma}$ converges in law to $X^{\infty}$ for the Meyer--Zheng topology, there exists an extraction $(\gamma_k)$ and $\J \subset \R_+$ of full Lebesgue measure such that the finite-dimensional distributions of $(X_t^{\gamma_k})_{t\in I}$ converge to those of $(X_t^{\infty})_{t\in \J}$. 

    Let $t>0$ be given, there exists $t_n$ elements of $\J$ converging to $t$ from the right. For $f$ a continuous bounded function on $E$, we now consider the following decomposition:
    \begin{align*}
    |\mathbb{E}_{\mu_0^{\gamma_k}}(f(X_{t}^{\gamma_k}) -\mathbb{E}_{\mu_0\Pa}(f(X_{t}^{\infty}))|
    &\leq |\mathbb{E}_{\mu_0^{\gamma_k}}(f(X_{t}^{\gamma_k})-\Pa f(X_{t}^{\gamma_k}))| \\ & \ \ \  \ \     +|\mathbb{E}_{\mu_0^{\gamma_k}}(\Pa f(X_{t}^{\gamma_k})-\Pa f(X_{t_n}^{\gamma_k}))| \\
      & \ \ \ \ \ \ \   +|\mathbb{E}_{\mu_0^{\gamma_k}}(\Pa f(X_{t_n}^{\gamma_k})) - \mathbb{E}_{\mu_0\Pa}(\Pa f(X_{t_n}^{\infty}))|\\ & \ \ \ \ \ \ \ \ \    +|\mathbb{E}_{\mu_0\Pa}( f(X_{t_n}^{\infty}))-\mathbb{E}_{\mu_0\Pa}(f(X_{t}^{\infty}))| 
\end{align*}
We remind that for almost all $\omega$, we have: $$f(X_{t_n}^{\infty}(\omega))\underset{n\to+\infty}{\longrightarrow} f(X_{t}^{\infty}(\omega))$$ by continuity of $f$ and right continuity of $s\mapsto X_s^{\infty}(\omega)$. Since all is trivially bounded, we have by dominated convergence: $$\mathbb{E}_{\mu_0\Pa}(f(X_{t_n}^{\infty}))\underset{n\to+\infty}{\longrightarrow} \mathbb{E}_{\mu_0\Pa}(f(X_{t}^{\infty})).$$
Using this last result and point ii) of Assumption \ref{ass:effgen}, we now consider $n$ such that $|\mathbb{E}_{\mu_0\Pa}(f(X_{t_n}^{\infty}))-\mathbb{E}_{\mu_0\Pa}(f(X_{t}^{\infty}))|\le\eps$ \, and \,  $(t_n-t) \ \underset{E}{\sup} |\La_{(0)}\Pa f|\leq \eps. $ We fix $n$ until the end of the proof.

We have thus bounded by $\eps$ the last term for all $k$, to conclude we only have to prove that the other terms are inferior to $\eps$ for $k$ big enough.

Since $t_n\in I$ we have $X_{t_n}^{\gamma_k}\underset{k\to+\infty}{\overset{\La}{\longrightarrow}} X_{t_n}^{\infty}$ and we bound the third term by $\eps$ for $k$ big enough.

We have for all $\eta>0$ that $d(E_{\eff},X_t^{\gamma_k})\le\eta$ with high probability for $k$ big enough. But $(f-\Pa f)_{|E_{\eff}}=0$, and Assumption~\ref{ass:P_cont} gives us that $f-\Pa f$ is continuous at points of $E_{\eff}$ with limit $0$, so taking $\eta$ small enough and using the compact containment hypothesis, we get that $|f(X_{t}^{\gamma_k})-\Pa f(X_{t}^{\gamma_k})|$ is small with a high probability for $k$ big enough. We can therefore bound the first term of by $\eps$ by taking $k$ big enough.

We now only have to bound the second term to conclude. Using Dynkin's formula (since $\La^{\gamma}\Pa f=\La_{(0)}\Pa f$) we have:
\begin{align*}|\mathbb{E}_{\mu_0^{\gamma_k}}(\Pa f(X_{t}^{\gamma_k})-\Pa f(X_{t_n}^{\gamma_k}))|&\leq \left| \int_t^{t_n} \mathbb{E}_{\mu_0^{\gamma_k}}\left(\La_{(0)} \Pa f (X_s^{\gamma}) \right) \right| \\ &\leq (t_n-t)\sup_{E} |\La_{(0)}\Pa f| \\ &\leq \eps.\end{align*}

Since any extraction of $X^{\gamma}$ still converges to $X^{\infty}$ for the Meyer--Zheng topology, we can make for every extraction a second one for which we have the convergence of the one dimensional distributions: we therefore have convergence in law of $X_t^{\gamma}$ to $X_t^{\infty}$ for all $t>0$.
\end{proof}
We will now invoke a technical Lemma that we will combine with the Markov property to prove the convergence of the finite dimensional distributions. 

\begin{Lem} \label{deuxdim}
    We assume as before that $X^{\gamma}$ starts from the initial condition $\mu_0^{\gamma}$ that converges in law towards a distribution denoted $\mu_0$. We furthermore assume here that the support of $\mu_0$ is included in $E_{\eff}$, so that $\mu_0\Pa=\mu_0$. Then, when $\gamma$ goes to $+\infty$, the couple $(X_0^{\gamma},X_t^{\gamma})$ converges in distribution towards the couple $(X_0^{\infty},X_t^{\infty})$. 
\end{Lem}
\begin{proof} We want to prove that for all $t>0$ and $f_0,f$ continuous bounded strictly positive functions on $E$, we have that:
    $$\mathbb{E}_{\mu_0^{\gamma}}(f_0(X_0^{\gamma})f(X_t^{\gamma}))\underset{\gamma\to+\infty}{\longrightarrow}\mathbb{E}_{\mu}(f_0(X_0^{\infty})f(X_t^{\infty}))$$
    We have that:
    \begin{align*}\mathbb{E}_{\mu_0^{\gamma}}(f_0(X_0^{\gamma})f(X_t^{\gamma}))&=\int_E f_0(z) \mathbb{E}_z(f(X_t^{\gamma}))\mu_0^{\gamma}(\mathrm{d}z) \\
    &=\mu_0^{\gamma}(f_0)\int_E  \mathbb{E}_z(f(X_t^{\gamma}))\nu_{\gamma}(\mathrm{d}z),\end{align*}
    where $\nu_{\gamma}(\mathrm{d}z)=\frac{1}{\mu_0^{\gamma}(f_0)} f_0(z)\mu_0^{\gamma}(\mathrm{d}z) $ is a probability measure on $E$. Since $\mu_0^{\gamma}(f_0)$ is uniformly minorated by a constant strictly superior to $0$ and converges to $\mu_{0}(f_0)$, we have that $\nu_{\gamma}$ converges in distribution towards $\nu(\mathrm{d}z)=\frac{1}{\mu(f_0)}f_0(z)\mu_{0}(\mathrm{d}z)$ that is a probability distribution on $E_{\eff}$ (so $\nu \Pa=\nu$). Using Proposition \ref{undim}, we have:
    \begin{align*}\mathbb{E}_{\mu_0^{\gamma}}(f_0(X_0^{\gamma})f(X_t^{\gamma}))=\mu_0^{\gamma}(f_0)\mathbb{E}_{\nu_{\gamma}}(f(X_t^{\gamma}))\underset{\gamma\to+\infty}{\longrightarrow} &\mu_{0}(f_0)\mathbb{E}_{\nu}(f(X_t^{\infty})) \\
    & \ \ \ \ \ \ \ \ =\mathbb{E}_{\mu_{0}}(f_0(X_0^{\infty})f(X_t^{\infty}))\end{align*}
\end{proof}
We now have all the tools needed to conclude.
\begin{proof}(Corollary \ref{conv_simple}) \\
We consider $0<t_1<\dots<t_n$ and $\eps\in]0,t_1[$, denoting $P_t^{\gamma}$ the semigroup associated with $X_t^{\gamma}$ and $P_t^{\infty}$ the semigroup associated with $X_t^{\infty}$, we notice that for all $f_1,\dots, f_n$ bounded strictly positives functions on $E$ we have:
$$\mathbb{E}_{\mu_0^{\gamma}}(f_1(X_{t_1}^{\gamma})\cdots f_n(X_{t_n}^{\gamma}))=P_{h_1}^{\gamma}(\rho_{\gamma},f_1(\cdot)P_{h_2}^{\gamma}(\cdot,f_2(\cdot) P_{h_3}^{\gamma}(\cdot,\dots P_{h_n}^{\gamma}(\cdot,f_n(\cdot))\dots))),$$
where $h_i=t_i-t_{i-1}-\eps$ with the convention $t_0=0$, and $\rho_{\gamma}$ is the law of $X_{\eps}^{\gamma}$ starting from $\mu_0^{\gamma}$. Proposition \ref{undim} gives us that $\rho_{\gamma}$ converges in distribution towards $\rho$, the law of $X_{\eps}^{\infty}$ starting from $\mu_0\Pa$ whose support is included in $E_{\eff}$.

To conclude we will prove by induction on $n$ that: $$\mathrm{d}z\mapsto P_{h_1}^{\gamma}(\rho_{\gamma},f_1(\cdot)P_{h_2}^{\gamma}(\cdot,f_2(\cdot) P_{h_3}^{\gamma}(\cdot,\dots f_{n-1}(\cdot)P_{h_n}^{\gamma}(\cdot, \mathrm{d}z)\dots)))$$ converges in distribution to: $$\mathrm{d}z\mapsto P_{h_1}^{\infty}(\rho,f_1(\cdot)P_{h_2}^{\infty}(\cdot,f_2(\cdot) P_{h_3}^{\infty}(\cdot,\dots f_{n-1}(\cdot)P_{h_n}^{\infty}(\cdot,\mathrm{d}z)\dots))),$$
both being positive finite non trivial measures since all the functions involved are positive bounded. Furthermore, the second measure is of support included in $E_{\eff}$.

Proposition \ref{undim} proves the convergence of the semigroup which is the result for $n=1$. We now take $n\geq 2$ and we assume that the result is proved at rank $n-1$. We write: $$\nu_{\gamma}(\mathrm{d}z)=P_{h_1}^{\gamma}(\rho_{\gamma},f_1(\cdot)P_{h_2}^{\gamma}(\cdot,f_2(\cdot) P_{h_3}^{\gamma}(\cdot,\dots f_{n-2}(\cdot)P_{h_{n-1}}^{\gamma}(\cdot,\mathrm{d}z)\dots)))$$
and:
$$\nu(\mathrm{d}z)=P_{h_1}^{\infty}(\rho,f_1(\cdot)P_{h_2}^{\infty}(\cdot,f_2(\cdot) P_{h_3}^{\infty}(\cdot,\dots f_{n-2}(\cdot) P_{h_{n-1}}^{\infty}(\cdot,\mathrm{d}z)\dots)))$$
By induction hypothesis, $\nu_{\gamma}$ converges to $\nu$ is distribution (so especially $\nu_{\gamma}(E)$ converges to $\nu(E))$. Writing $\overline{\nu_{\gamma}}=\frac{1}{\nu_{\gamma}(E)}\nu_{\gamma}$ and $\overline{\nu}=\frac{1}{\nu(E)}\nu$ we have that $\overline{\nu_{\gamma}}$ is a probability measure converging to $\overline{\nu}$. Hence, Lemma \ref{deuxdim} gives us that for all continuous bounded function $f$ we have:
\begin{align*} & P_{h_1}^{\gamma}(\rho_{\gamma},f_1(\cdot)P_{h_2}^{\gamma}(\cdot,f_2(\cdot) P_{h_3}^{\gamma}(\cdot,\dots f_{n-1}(\cdot)P_{h_{n}}^{\gamma}(\cdot,f)\dots))) \\ &=\nu_{\gamma}(E)\mathbb{E}_{\overline{\nu_{\gamma}}}(f_{n-1}(X_0^{\gamma})f(X_{h_n}^{\gamma})) \\ &\underset{\gamma\to+\infty}{\longrightarrow} \nu(E)\mathbb{E}_{\overline{\nu}}(f_{n-1}(X_0^{\infty})f(X_{h_n}^{\infty}))\\
&=P_{h_1}^{\infty}(\rho,f_1(\cdot)P_{h_2}^{\infty}(\cdot,f_2(\cdot) P_{h_3}^{\infty}(\cdot,\dots f_{n-1}(\cdot) P_{h_{n}}^{\infty}(\cdot,f)\dots))).
\end{align*}
Since this is true for all bounded functions $f$, this proves the result at rank $n$.

Hence, we take $f=f_n$ and we get that:
$$\mathbb{E}_{\mu_0^{\gamma}}(f_1(X_{t_1}^{\gamma})\cdots f_n(X_{t_n}^{\gamma}))\underset{\gamma\to+\infty}{\longrightarrow}\mathbb{E}_{\mu_0\Pa}(f_1(X_{t_1}^{\infty})\cdots f_n(X_{t_n}^{\infty}))$$
\end{proof}

\subsection{Proof of Proposition~\ref{cvunif}}

We propose in this section a proof of Proposition~\ref{cvunif} (rewritten as Lemma~\ref{the:uniform_theorem} below) which states that the large time convergence of the dominant process towards $E_{\eff}$ is uniform with respect to starting points taken on compacts. 

Let us reformulate Proposition~\ref{cvunif} by making assumptions explicit:
\begin{Lem}
\label{the:uniform_theorem} Let $E$ be Polish, $E_{\eff}$ a closed subset of $E$ and assume that $(x,t,\omega) \mapsto X_{t}(x)(\omega) \in E$ be a measurable map defined for $(t,x) \in \R_+ \times E$ and $\omega$ in a filtered probability space. Assume that $t \mapsto X_t(x)$ is for all $x\in E$ a c\`adl\`ag Markov process for the considered filtration. We assume moreover that:
\begin{itemize}
    \item For all $t \geq 0$, $x \mapsto X_t(x)$ is continuous in probability (Feller-type continuity).
    \item For all $x\in E$, the process $(X_t(x))_{t\ge 0}$ converges in probability towards $X_{\infty}(x)\in E_{\eff}$.
    \item $x\mapsto \Pa(x,\cdot)=\Law(X_\infty(x))$ is continuous for the topology of convergence in law.
    \item We have $\Pa(x,\cdot)=\delta_z$ for $z\in E_{\eff}$ if and only if $x=z\in E_{\eff}$
    \item We have the compact containment~\eqref{eq:cont_dom}: for $K$ compact and $\eps>0$ given, there exists a compact $K_{\eps}$ such that:
\begin{equation*}
\sup_{x\in K}\P \b{\exists t\ge0, \,  X_t(x) \in K_{\eps} ^c } \leq \eps. 
\end{equation*}
\end{itemize}
Then there exists a continuous bounded non-negative function $f$ on $E$ with $f > 0$ on $E \setminus E_\eff$ and $f=0$ on $E_{\eff} $ such that we have for all $K\subset E$ compact:
\begin{equation} \label{convergence}
\lim_{t \to + \infty} \sup_{x \in K}\E\b{f(X_t(x))}=0.
\end{equation}
\end{Lem}

To prove the Lemma, we start with a definition of an object we will use in the proof. We say for $\varepsilon>0$ that $E_{\eff}$ admits a \emph{measurable $\varepsilon$-projector} in $E$ if there exists a measurable function $p_{\varepsilon}:E\to E_{\eff}$ such that for all $z\in E$ we have $d(z,E_{\eff})+\varepsilon\ge d(z,p(z))$. We have the following result on the existence of $\varepsilon$-projector.

\begin{Lem}[Existence of an $\eps$-projector]\label{proj}
    Let $(E,d)$ be a Polish space, $F$ a closed subset of $E$ and $\varepsilon>0$. Then $F$ admits a (measurable) $\varepsilon$-projector.
\end{Lem}
\begin{proof}
    Let $$g:y\mapsto d(y,F)=\inf_{z\in E_{\eff}} d(y,z).$$
    The function $g$ is continuous and therefore measurable. We now consider the multifunction $\psi$ on $E$ defined by:$$\psi:y\mapsto \overline{B}(y,g(y)+\varepsilon).$$
    Then for $U$ an open ball of center $x$ and of radius $\alpha$, we have that \begin{align*}\{y\in E, \, \psi(y)\cap U\neq \emptyset\}&=\{y\in E, \, d(y,x)< \alpha+g(y)+\varepsilon\} \\ &=\{y\in E, \, d(y,x)-g(y)<\alpha+\varepsilon\}\end{align*}

    Since $y\mapsto d(y,x)-g(y)$ is measurable, we have that this set is in $\mathcal{B}(E)$. Using that all open set of $E$ are countable union of balls, we have that the result stays true for all $U$ open set of $E$. Eventually, Kuratowski–Ryll-Nardzewski measurable selection theorem gives us that $\psi$ admits a measurable selection, that is by definition an $\varepsilon$-projector
\end{proof}

To prove Lemma~\ref{the:uniform_theorem} we will need to use the following general lemma, which proves, using the continuity with respect to $x$ of $X_t(x)$ in probability and of $X_\infty(x)$ in law, the continuity of $X_\infty(x)$ in probability, up to some compact containment.
\begin{Lem} \label{lem:approx_limit} Assume $E$ is Polish. Let $(x,t,\omega) \mapsto X_{t}(x)(\omega) \in E$ be a measurable map defined for $(t,x) \in \R_+ \times E$ and $\omega$ in a filtered probability space. Assume that $t \mapsto X_t(x)$ is for all $x\in E$ a c\`adl\`ag Markov process for the considered filtration. We assume moreover that:
\begin{itemize}
    \item For all $t \geq 0$, $x \mapsto X_t(x)$ is continuous in probability. 
    \item For all $x\in E$, the process $(X_t(x))_{t\ge 0}$ converges in probability towards $X_{\infty}(x)\in E_{\eff}$  whose distribution is given by $\Pa(x, \, . \,)$.
    \item $x\mapsto \Pa(x,\cdot)=\Law(X_\infty(x))$ is continuous for the metric topology of convergence in law.
\end{itemize}

Then for all $K^+ \subset E$ and $K \subset K^{+}$ compact sets, all $x\in K$ and all $\eps,h>0$, there exists $\eta=\eta_{K,K^+,x,\eps,h}>0$ such that for all $x'\in B(x,\eta) \cap K$ ,  we have:
$$\mathbb{P}\left(\d(X_{\infty}(x),X_{\infty}(x')) > h \right)\leq \frac{\eps}{16}+ 2 \mathbb{P}\p{\exists t \geq 0, \, X_t(x') \notin K^+ }.$$

\end{Lem}
\begin{proof} The proof of Lemma~\ref{lem:approx_limit} is decomposed in several steps.

\paragraph{Two uniform continuity estimates.} First, by assumption, the function $x\mapsto \Pa(x,\cdot)= \Law(X_\infty(x))$ is continuous for the metric topology of convergence in law. Since $K^{+}$ is compact, Heine's theorem applies and gives us that it is uniformly continuous on $K^{+}$. Remark also that for $z\in E_{\eff}$ one has $\Pa(z,\cdot)=\delta_z$. As a direct consequence, for all $\alpha,h>0$, there exists $a_{\alpha,h}>0$ such that for all $z\in E_{\eff}\cap K^{+}$ and $x\in K^{+}$ : \begin{equation} \label{approx}
    d(x,z)\leq a_{\alpha,h}  \ \ \ \ \ \text{implies} \ \ \ \ \ \ \mathbb{P}(d(X_{\infty}(x),z) \leq \frac{h}{2}))\geq 1-\alpha
\end{equation} 

 Second, let $t \geq 0$ be any given time, and let us consider the continuity with respect to the initial conditions. By assumption, the map $x \mapsto X_t(x)$ is continuous for the metric topology of convergence in probability of random variables. Since $K$ is compact, Heine's theorem applies again and yields uniform continuity, which can be expressed as follows. For all $a,\alpha > 0$, there exists $\eta_{t,a,\alpha}>0$ such that:
 \begin{equation} \label{approx2}
    d(x,x')\leq \eta_{t,a,\alpha} \ \ \ \ \ \text{implies} \ \ \ \ \ \ \mathbb{P}\left(d(X_t(x),X_t(x'))\le\frac{a}{2}\right)\geq 1 - \alpha
\end{equation}

\paragraph{A first estimate.} Let $\alpha,a>0$ and $x \in K$ be given. Considering $p_{a/4}$ an $a/4$-projector whose existence in given by Lemma \ref{proj}, we claim that there exists $ T=T_{x,a,\alpha}$ and $\eta=\eta_{x,a,\alpha} > 0$ such that for any $x'\in B(x,\eta)\cap K$, we have: \begin{align*}\mathbb{P}(d(X_{T}(x'), p_{a/4}(X_{T}(x))) \leq a) \geq 1-2\alpha. \numberthis \label{et1}\end{align*}

Indeed, we first have by assumption that $(X_t(x))_{t \geq 0}$ converges in probability to $X_{\infty}(x)$, therefore there exists $T = T_{x,a,\alpha}>0$ such that  \begin{equation} \label{eq:T}
    \P(d(X_{T}(x),E_\eff) \leq a/4 ) \geq \mathbb{P}(d(X_{T}(x),X_\infty(x)) \leq a/4 ) \geq  1-\alpha .
\end{equation}

Next, we can consider the lower bound
\begin{align*}&\mathbb{P}\p{ d(X_T(x'),p_{a/4}(X_{T}(x) ) ) \leq a } \\&\geq \mathbb{P}\left( d(X_T(x'),X_T(x)) \leq \frac{a}{2}, \, \, X_T(x)\in B(p_{a/4}(X_T(x)),\frac{a}{2})\right) \\
&\geq \mathbb{P}\left( d(X_T(x'),X_T(x)) \leq \frac{a}{2}, \, \, X_T(x)\in B(E_{\eff},\frac{a}{4})\right),  \end{align*}
and using first~\eqref{eq:T} and then~\eqref{approx2} with $\eta=\eta_{T,a,\alpha}$ we obtain the claim~\eqref{et1}.

\paragraph{Main estimate.} The key of the proof consists in the following claim. Let $K \subset K^+ \subset E$ compact subsets, $\alpha \in [0,1]$, $h>0$, and $x \in K$ be given. Consider the event $\calB_{x'} = \set{X_t(x') \in K^+, \, \forall t \geq 0}$. We claim that there exists  $\eta > 0$ and $T \geq 0$ such that for any $x'\in B(x,\eta) \subset K $, we have:
\begin{align*}
    \mathbb{P}(X_{\infty}(x')\in B( p_{a/4}(X_{T}(x)),\frac{h}{2})) \geq 1- 5\alpha-\mathbb{P}(\calB_{x'}^c).
\end{align*}
This estimates will subsequently quite easily yield the proof of the whole lemma.

The key to prove the claim consists in the following conditioning:
\begin{align*}
    &\mathbb{P}(X_{\infty}(x')\in B(p_{a/4}(X_{T}(x)),\frac{h}{2}))\\ &\geq \underbrace{ \mathbb{P}(X_{\infty}(x')\in B(p_{a/4}(X_{T}(x)),\frac{h}{2})\left| X_T(x')\in B(p_{a/4}(X_{T}(x)),a)\cap K^{+} \right.)}_{\text{ii)}} \\
    & \ \ \ \ \ \ \ \ \ \ \times \underbrace{\mathbb{P}(X_T(x')\in B(p_{a/4}(X_{T}(x)),a)\cap K^{+})}_{\text{i)}} \\
\end{align*}
Using our first estimate~\eqref{et1}, there exist a $T >0$ and a $\eta >0$ such that for all $x' \in B(x,\eta) \cap K $ the term i) can be lower bounded as follows:
\begin{align*}
    & \text{i)}=\mathbb{P}(X_T(x')\in B(p_{a/4}(X_{T}(x)),a)\cap K^{+}) \\
    &\geq 1-\mathbb{P}(X_T(x')\notin B(p_{a/4}(X_{T}(x)),a))-\mathbb{P}(\calB_{x'}^c)\\
    &\geq 1-2\alpha-\mathbb{P}(\calB_{x'}^c)).
\end{align*}
Then, since $t \mapsto (X_t(x),X_t(x'))$ is homogeneous in time and satisfies the Markov property, the uniform estimate~\eqref{approx} (recall that $x' \in K \subset K^+$) directly implies that ii) can be lower bounded by 
\begin{align*}\text{ii)} &= \mathbb{P}\left(X_{\infty}(x')\in B(p_{a/4}(X_{t}(x)) ,\frac{h}{2})\left| X_t(x')\in B(p_{a/4}(X_{t}(x)),a)\cap K^{+}\right.\right) \\ &\geq 1-\alpha  \label{et2} \numberthis \end{align*}
for any deterministic $t \geq 0$  and in particular for $t=T$.

Finally we obtain that $\text{i)} \times \text{ii)} \geq (1-\alpha)\times (1-2\alpha-\mathbb{P}(\calB_{x'}^c)) \geq 1-5\alpha-\mathbb{P}(\calB_{x'}^c)$.


\paragraph{Final remark.} Eventually, one can apply two times the previous claim and get for all $x'\in B(x,\eta)$
\begin{align*}
    &\mathbb{P}(d(X_{\infty}(x'),X_{\infty}(x)) \leq h)\\ &\geq \mathbb{P}(X_{\infty}(x')\in B(p_{a/4}(X_{T}(x)), \frac{h}{2}) \ \& \ X_{\infty}(x)\in B(p_{a/4}(X_{T}(x)),\frac{h}{2})) \\
    & \geq 1 - \mathbb{P}(X_{\infty}(x')\notin B(p_{a/4}(X_{T}(x)),\frac{h}{2})) - \mathbb{P}(X_{\infty}(x)\notin B(p_{a/4}(X_{T}(x)),\frac{h}{2})) \\
    &\geq 1-10\alpha-2\mathbb{P}(\calB_{x'}^c).
\end{align*}
We then choose $\alpha$ small enough so that $10\alpha\le\frac{\varepsilon}{16}$, and we have proved the lemma.
\end{proof}

Before carrying out the proof of the theorem, we end with a usual technical lemma that will be useful in the proof. We recall as a preamble that the distance function to any closed set $C$ is $1$-Lipschitz since:
\begin{align*}
   \underset{y\in C}{\inf} d(x',y)-\underset{y\in C}{\inf} d(x,y)\leq \underset{y\in C}{\sup}\p{ d(x',y)-d(x,y) }\leq \underset{y\in E_{\eff}}{\sup} d(x,x')  =  d(x,x') .
\end{align*}

\begin{Lem}\label{lem:convdet} Let $(E,\rm{d})$ be a separable metric space. There exists a countable family of $1$-Lipschitz, bounded by $1$ functions that is convergence determining.
\end{Lem}
\begin{proof}
Since $E$ is metric separable, it is second countable and there exists a countable family $\calB$ of closed sets that contains a decreasing sequence whose intersection is any given closed set. Let us consider the following countable family of bounded and Lipschitz functions

$$ f_{n,C}(x) = 1 - \min\p{n d(x,C),1},  \qquad C \in \calB, \, n \in \N\setminus\{0\} .$$ 

Let $C$ be a given closed set and $C_p \searrow C$ a decreasing sequence converging to $C \in \calB$. One has then $\inf_{p,n}  f_{n,C_p}(x) = \one_{C}(x)$. Let $(\mu_q)_{q}$ denotes a sequence of probabilities, and $\mu$ another probability such that $\limsup_{q} \mu_q(f_{n,C_p}) = \mu(f_{n,C_p})$ for any $n,p$. Obviously, by monotone convergence
$$
 \limsup_{q} \mu_q(C) = \limsup_{q}  \inf_{p,n} \mu_q(f_{n,C_p}) \leq     \inf_{p,n}  \limsup_{q} \mu_q(f_{n,C_p}) = \mu(C).
$$
This implies by portmanteau lemma that $\lim_q \mu_q = \mu$ in distribution.

Taking $g_{n,C_p}=f_{n,C_p}/n$ that is $1$-Lipschitz and bounded by $1$, we have that $\limsup_{q} \mu_q(f_{n,C_p}) = \mu(f_{n,C_p})$ if and only if $\limsup_{q} \mu_q(g_{n,C_p}) = \mu(g_{n,C_p})$, so the family $(g_{n,C_p})_{n\ge 1,C\in \mathcal{B}}$ verifies the claim.

\end{proof}

\begin{proof}[Proof of Lemma \ref{the:uniform_theorem}]
We consider a compact $K$ of $E$ and $\eps$. By compact containment (Assumption~\ref{ass:tight}), there exists a compact $K_{\eps}$ such that denoting $$
\calB_x=\{X_t(x)\in K_{\eps}, \, \forall t\ge 0\}
$$ 
we have $\sup_{x\in K} \mathbb{P}(\calB_x^c)\le \frac{\eps}{8} $. 

We take $(f_n)_{n\ge1}$ a family of convergence determining functions on $E$ that are bounded by $1$ and $1$-Lipchitz for the distance $\min(1,d(x,y))$ (Lemma~\ref{lem:convdet}). For $x,y\in E$ the following (harmonic) semi-distance:
$$d_{\text{har}}(x,y)=\sum_{n\geq 1} |\Pa f_n(x)-\Pa f_n(y)|2^{-n}.$$
Since $(f_n)_n$ is separating; we note that $d_{\text{har}}(x,y)=0$ if and only if  $\Pa(x,\cdot) = \Pa(y,\cdot))$. Moreover the continuity assumption on $\Pa$ implies that $d_{\text{har}}$ is topologically weaker than $d$.

To prove Lemma \ref{the:uniform_theorem}, we consider the harmonic distance function to $E_\eff$ defined by
$$f:x\in E\mapsto \underset{y\in E_{\eff}}{\inf} d_{\text{har}}(x,y) \in \R_{+}.$$
We claim that $f$ is null on $E_{\eff}$, strictly positive on $E\setminus E_{\eff})$ and $1$-Lipschitz for $d_{\text{har}}$ (hence continuous for $d$). We will subsequently prove~\eqref{convergence} for this function. First, we have trivially that $f_{|E_{\eff}}=0$, and that $f$ is $1$-Lipschitz as a distance function. Let us now prove that we have $f_{|E\setminus E_{\eff}}>0$. Let $x \in E\setminus E_{\eff}$ be given. We assume that $f(x)=0$ and we prove that $x\in E_{\eff}$. If $f(x)=0$, we can find by definition of $f$ a sequence $y_n\in E_{\eff}$ such that $d_{\rm{har}}(y_n,x)\underset{n\to+\infty}{\longrightarrow} 0$, so $\langle \delta_{y_n},f_p\rangle=f_p(y_n)\underset{n\to+\infty}{\longrightarrow}\Pa f_p(x)=\langle \Pa(x,\cdot),f_p\rangle$ for all $p\ge1$. Since the $f_p$ are convergence determining, we thus have that $\delta_{y_n}$ converges weakly towards $\Pa(x,\cdot)$, but a converging sequence of Dirac distributions of support included in a closed set can only converge towards another Dirac measure of support included in the same closed set, so eventually there exists $y\in E_{\eff}$ such that $\Pa(x,\cdot)=\delta_y=\Pa(y,\cdot)$. By Assumption~\ref{ass:non_dirac}, we have therefore $x=y\in E_{\eff}$.

We can develop the main argument, using the first key fact that for all $n\geq 1$ and $x\in E$, the process $t\mapsto \Pa f_n(X_t(x))$ is a martingale for the underlying filtration, so that the absolute value:
$$
t \mapsto |\Pa f_n(X_{ t}(x))-\Pa f_n(X_{ t}(x'))|
$$
is a sub-martingale whose expectation is thus increasing with time. Note also that $\Pa f_n(X_{\infty}(x'))=f_n(X_{\infty}(x'))$ almost surely since $X_{\infty}(x')\in E_{\eff}$. The second key fact is the conclusion of Lemma \ref{lem:approx_limit}, which is to be used for $K=K$,\, $K^{+}=K_{\eps}$, \, $x=x$, $h=\frac{\eps}{8}$, a well-chosen $\eta_x \equiv \eta_{K, K_\eps,x,\eps}$ and all $x'\in B(x,\eta_x)\cap K$.

We then proceeds to estimate $\mathbb{E}(f(X_t(x'))$ as follows:
\begin{align*}
    &|\mathbb{E}(f(X_t(x'))|\\
    &\le|\mathbb{E}(f(X_t(x')) \\
    &=|\mathbb{E}(\left(f(X_{ t}(x'))-f(X_{\infty}(x'))\right)|\\
    &\leq\mathbb{E}(d_{\text{har}}(X_{\infty}(x'),X_{ t}(x'))) \\ 
    &\leq \sum_{n\geq 1} 2^{-n} \mathbb{E}\left(|\Pa f_n(X_{ t}(x))-\Pa f_n(X_{ t}(x'))|\right)\\& \quad  \quad \quad +2^{-n}\mathbb{E}\left(|\Pa f_n(X_{ t}(x))-\Pa f_n(X_{\infty}(x))|\right) \\& \quad \quad \quad \quad \quad \quad +2^{-n}\mathbb{E}\left(|\Pa f_n(X_{\infty}(x))-\Pa f_n(X_{\infty}(x'))|\right) \\
    &\leq \sum_{n\geq 1} 2^{-n}\mathbb{E}\left(|\Pa f_n(X_{ t}(x))-\Pa f_n(X_{\infty}(x))|\right)\\& \quad \quad \quad +2^{-n}\mathbb{E}\left(|\Pa f_n(X_{\infty}(x))-\Pa f_n(X_{\infty}(x'))|\right)\times 2\\
    &\leq \sum_{n\geq 1} 2^{-n}\mathbb{E}\left(|\Pa f_n(X_t(x))-\Pa f_n(X_{\infty}(x))|\right)\\& \quad \quad \quad +2^{-n}\mathbb{E}\left(|f_n(X_{\infty}(x))- f_n(X_{\infty}(x'))|\right)\times 2 \\
    &\leq \sum_{n\geq 1} 2^{-n}\mathbb{E}\left(|\Pa f_n(X_t(x))-\Pa f_n(X_{\infty}(x))|\right)\\& \quad \quad \quad +2^{-n}\mathbb{E}\left(  d \left(X_{\infty}(x),X_{\infty}(x') \right)\wedge 1 \right)\times 2 \\
    &\leq \sum_{n\geq 1} 2^{-n}\mathbb{E}\left(|\Pa f_n(X_t(x))-\Pa f_n(X_{\infty}(x))|\right)\\& \quad \quad \quad +2^{-n}(\frac{\eps}{8}+\mathbb{P}(d(X_{\infty}(x),X_{\infty}(x'))>\frac{\eps}{8}))\times 2 \\
    &\leq \sum_{n\geq 1} 2^{-n}\mathbb{E}\left(|\Pa f_n(X_t(x))-\Pa f_n(X_{\infty}(x))|\right)\\& \quad \quad \quad +2^{-n}(\frac{3\eps}{16}+2\mathbb{P}(\calB_x^c)))\times 2 \\
    &\leq \sum_{n\geq 1} 2^{-n}\mathbb{E}\left(|\Pa f_n(X_t(x))-\Pa f_n(X_{\infty}(x))|\right)\\& \quad \quad \quad + \frac{7\eps}{8}
\end{align*}
Using that the $f_n$ are bounded by $1$, the first term above converges to $0$ when $t$ approaches infinity, and is thus inferior to $\frac{\eps}{8}$ for all $t>T_x$ for some $T_x$.

By covering $K$ by a finite number of balls (we can because $K$ is compact) of centers $x_i$ radius $\eta_{x_i}$ for $i=1 \ldots I$, we find that for all $t\geq\max_i T_{x_i}$ it holds
$$\sup_{x' \in K}\E\b{f(X_t(x'))}\le \eps.$$
Since this is true for all $\eps>0$, we have eventually:
$$\lim_{t \to + \infty} \sup_{x \in K}\E\b{f(X_t(x))}=0.$$
\end{proof}

\section{Examples}\label{sec:examples}
We will now use Theorem \ref{main_theorem} and Corollary \ref{conv_simple} on concrete examples to get homogeneization results. Once the framework is settled, we only have to check the common hypothesis of both Theorem \ref{main_theorem} and  Corollary \ref{conv_simple} to have the convergence in law respectively in the sense of the Meyer--Zheng topology and in the sense of the finite-dimensional distributions.  

\subsection{Diffusion in a simplex (see~\cite{faure2022averaging})}\label{sec:simplex}
Consider $E$ a closed simplex in $\R^n$, that is the non-degenerate intersection of $n+1$ affine half-spaces $E_i$, $i=1 \ldots n+1$. $E_{\eff}$ is defined to be the $n+1$ vertices of $E$: we have therefore that $E$ is compact and $E_{\eff}$ is a finite set. We take $\mathcal{D}_{\eff}=\mathcal{C}_c(E_{\eff})$ which is here the set of all real functions on $E_{\eff}$. We assume in this example that the dominant process is a martingale and that the subdominant process is general, we write:
$$\mathrm{d}X_t^{\gamma}=\sqrt{\gamma}\sigma\left(X_t^{\gamma}\right)\mathrm{d}W_t+\sigma_0\left(X_t^{\gamma}\right)\mathrm{d}B_t+b\left(X_t^{\gamma}\right)\mathrm{d}t,$$
where we assume that $\sigma, \sigma_0: O \to \mathcal{M}_n(\R)$ and $b: O \to \R^n$ are Lipschitz continuous in $E$ (for the Euclidean distance). Thus we have:
\begin{equation*}
     \La_{(0)}=\langle b, \nabla_x\rangle+\frac{1}{2}\langle \sigma_0\sigma_0^{\dagger} \nabla_x, \nabla_x\rangle \quad \text{and} \quad  \La_{(1)}=\frac{1}{2}\langle  \sigma\sigma^{\dagger} \nabla_x, \nabla_x\rangle
\end{equation*} 

In order to obtain an appropriate behaviour of the above process at the boundary $\partial E$, it is necessary to add several assumptions:
\begin{enumerate}[label=(\roman*)]
    \item For all $x \in \partial E$, $\sigma(x)$ and $\sigma_0(x)$ have their images in the tangent space of $\partial E_i$ for each $i$ with $x \in \partial E_i$. In particular $\sigma(x) = 0$ if $x \in E_\eff$.
    \item For all $x \in \partial E$, $b(x)$ vanishes or points to the interior of $E$.
    \item $\sigma(x) = 0$ if and only if $x \in E_\eff$.
\end{enumerate}

The above special assumptions ensure in particular that the process stays in $E$ almost surely for any $\gamma>0$ and $t\geq 0$ (see~\cite{faure2022averaging}).
\begin{Lem} Assume that in addition to Lipschitz continuity, $\sigma$, $\sigma_0$ and $b$ satisfy i) and ii) above. Then $X_t^{\gamma}$ belongs to $E$ for all time $t$.
\end{Lem}
As a consequence, according to Section~\ref{sec:diff}, we therefore have Assumption~\ref{ass:flow} and Assumption~\ref{ass:contper}.

It has also been proved in~\cite{faure2022averaging}, Remark~6, that for all $x\in E$:
$$\Pa(x,\cdot)=\sum_{z\in E_{\eff}} H_z(x)\delta_z,$$
where $x\mapsto H_z(x)$ is the only affine function on $E$ that is: i) equal to $H_z(z)=1$ at $x=z$, and ii) $H_z(z')=0$ for any $z'\in E_{\eff}\setminus \{z\}$. In short, the proof works as follows: since the dominant process is a bounded martingale, its quadratic variation is almost surely finite so that the dominant process converges almost surely to a point where $\sigma$ is null; by iii) above the latter is a point of $E_{\eff}$. $X_t^{(1,0)}$ converges in probability towards $E_{\eff}$ when $t\to+\infty$. Since $H_z$ is affine, it holds that $\La_{(1)} \Pa = 0$ and $\Pa$ is continuous at $E_\eff$ with limit the identity; a standard application of It\^o formula implies then that indeed $\lim_{t \to +\infty} \Law(X_t^{(1,0)}(x)) = \Pa(x, \, . \,)$ and Assumption~\ref{ass:main} is checked. Furthermore, the expression of $\Pa$ gives us Assumption \ref{ass:non_dirac}.

We will next prove again the main result of~\cite{faure2022averaging}, that is:

\begin{The}
    The process $X^{(\gamma,1)}$ converges in law (for the pseudo-paths topology as well as for finite dimensional time marginals) when $\gamma$ approaches infinity towards a continuous-time Markov chain $X^\infty$ on $E_{\eff}$ of transition matrix $\La_{\eff}=\left(b(x)\cdot \nabla_x H_z(x) \right)_{x,z\in E_{\eff}}$.
\end{The}

We just have to check the four other hypothesis.

The assumption \ref{ass:tight} is trivially verified since $E$ is a compact.

The Assumption \ref{ass:P_cont} and Assumption \ref{ass:effgen} are easily proved using the fact that $\Pa$ is affine: $\Pa \varphi$ being smooth on $E$ for all $\varphi\in\mathcal{D}_{\eff}$, point $ii)$ is a direct consequence of Itô's formula applied on $\Pa \varphi(X^{(\gamma,1)})$. Indeed, for all $\varphi\in\mathcal{D}_{\eff}$, we have:
\begin{align*}
    d(\Pa \varphi(X_t^{(\gamma,1)})&=(\gamma\La_{(1)}+\La_{(0)})(\Pa \varphi)(X_t^{(\gamma,1)})\mathrm{d}t+\nabla_x(\Pa \varphi)(X_t^{\gamma})\cdot \sigma_0(X_t^{(\gamma,1)})\mathrm{d}B_t \\ & \ \ \ \ \ \ \ \ \ \ \ \  +\sqrt{\gamma}\nabla_x(\Pa \varphi)(X_t^{\gamma})\cdot \sigma(X_t^{(\gamma,1)})\mathrm{d}W_t,
\end{align*}
integrating this inequality and using that $\La_{(1)}\Pa \varphi=0$, we get eventually:
\begin{align*}&\Pa \varphi(X_t^{(\gamma,1)}(x))-\int_0^t \La_{(0)}\Pa \varphi(X_s^{(\gamma,1)}(x))\mathrm{d}s \\ &=\Pa\varphi(x)+\int_0^t\nabla_x(\Pa \varphi)(X_t^{(\gamma,1)}(x))\cdot \sigma_0(X_t^{(\gamma,1)}(x))\mathrm{d}B_s \\ & \ \ \ \ \ +\int_0^t\sqrt{\gamma}\nabla_x(\Pa \varphi)(X_t^{(\gamma,1)}(x))\cdot \sigma(X_t^{(\gamma,1)})\mathrm{d}W_s,
\end{align*}
with the second term clearly being a martingale for the natural filtration of $X_t^{(\gamma,1)}(x)$.

For point $iii)$, using the explicit expression of $\Pa \La_{(0)}$ and the continuity of $b$ and $\nabla_x H_z$, we have for $z_0\in E_{\eff}$ that: $$\lim_{x \to z_0} \La_{(0)} \Pa \b{\ph}(x)=\lim_{x\to z_0} \sum_{z\in E_{\eff}}\varphi(z)b(x)\cdot\nabla_x H_z(x)=\sum_{z\in E_{\eff}} \varphi(z) b(z_0)\cdot \nabla_x H_z(z_0). $$
Identifying $\varphi$ with the vector $(\varphi(z))_{z\in E_{\eff}}$ we have eventually that $\La_{\eff} \varphi=\left(b(x)\cdot \nabla_x H_z(x) \right)_{x,z\in E_{\eff}} \varphi$, which is trivially bounded measurable.

Since $b(x_0)\cdot \nabla_x H_z(x_0)\leq 0$ and $b(x_0)\cdot \nabla_x H_z(x_0)\geq 0$ for $z\neq x_0$, using furthermore that $\sum_{z\in E_{\eff}} \limits b(z)\cdot \nabla_x H_z(x_0)=0$, we have that $\La_{\eff}$ is the generator of a continuous-time Markov process on $E_{\eff}$. The martingale problem is therefore well-posed and Assumption \ref{ass:wellposed} is checked.

\subsection{Diffusion in a strip}

\label{sec:exa_strip}
Define $E = [-1,1] \times \R$ and $E_{\eff}$ its border $\set{-1,1} \times \R$. We take  $\calD_{\mrm{eff}}=C_c^\infty\left(\{-1,1\}\times\R\right)$
as space of test functions. In this subsection we will study the process $X^{\gamma}=X^{(\gamma,1)}=(Y^{\gamma},Z^{\gamma})$ defined by:
$$
\begin{cases}
& \mathrm{d} Y_t^{\gamma} = \sqrt{\gamma}\sigma(Y_t^{\gamma},Z_t^{\gamma}) \, \mathrm{d} W_t+b(Y_t^{\gamma},Z_t^{\gamma}) \mathrm{d} t\\
& \mathrm{d} Z_t^{\gamma} = \gamma\sigma(Y_t^{\gamma},Z_t^{\gamma})^2 \, \mathrm{d} t = \d [Y^\gamma]_t.
\end{cases}
$$

where $b:[-1,1]\times \R\to \R$ is a Lipschitz bounded function such that $b(1,z)\leq 0$ and $b(-1,z)\geq 0$ and $\sigma:[-1,1]\times\R \to \R$ is Lipschitz bounded function with $\sigma^{-1}(\{0\})=\{\pm 1\}\times \R$. Note that when the value of $b$ does not depend of $z$-coordinate, the above process is the simplex example of Section~\ref{sec:simplex} for $n=1$ (the variable $Y$) with the addition of its quadratic variation (the variable $Z$). Adding the quadratic variation is of special interest to keep track of the ''intrinsic diffusive time'' of the dominant process through the homogenization procedure. Since our process is strong solution with Lipschitz coefficients and $Y$ remains in  $[-1,1]$ for all time, according to Section~\ref{sec:diff}, we have Assumption~\ref{ass:flow} and Assumption~\ref{ass:contper}.

The main homogenization theorem is then following.
\begin{The}
    The process $X^{(\gamma,1)}$ in $[-1,1] \times \R_+$ converges in law (pseudo-path topology and finite-dimensional) when $\gamma$ approaches infinity towards a Levy-type process of generator: 
    \begin{align*}
    \calL_{\eff} ( \ph)(y,z) =& |b|(y,z) \int_{\R_+} \ph(y,z+h) - \ph(y,z)  \d \mu(h) \\ 
    & \qquad \qquad  +  |b|(y,z) \int_{\R_+} \ph(-y,z+h) - \ph(y,z)  \d \nu(h)  \\ 
    \end{align*}
    in which $\nu$ is a finite positive measure of total mass $\nu(\R_+) = 1/2$, and $\mu$ positive Levy measure of a subordinator ($\int_{\R_+} \min(1,t) \mu(\d t) < + \infty$). The latter are given by the explicit expansion \\ $\mu(\mathrm{d}t)=\left(\frac{\pi^2}{16} \displaystyle\sum_{k=-\infty}^{+\infty} k^2 e^{-t\frac{\pi^2 k^2}{8}} \right) \mathrm{d}t$ and $\nu(\mathrm{d}t)=\left(-\frac{\pi^2}{16} \displaystyle\sum_{k=-\infty}^{+\infty} (-1)^k k^2 e^{-t\frac{\pi^2 k^2}{8}} \right) \mathrm{d}t$.
\end{The}

We have to check the six remaining hypothesis. 

On the one hand we have that $(Y^{(1,0)}_t)_t$ is a bounded martingale that converges almost surely to $Y^{(1,0)}_{\infty}\in \{\pm 1\}$. On the other hand, since $Y^{(1,0)}_t=\int_0^t \sigma(Y_t^{(1,0)},Z_t^{(1,0)}) \mathrm{d}W_s$, we have using Itô isometry and $|Y_t|\leq 1$ that:
$$\int_0^{t} \mathbb{E}_{(y,z)}\left(\sigma(Y_t^{(1,0)},Z_t^{(1,0)})^2 \right) \mathrm{d}s=\mathbb{E}_{(y,z)}\left(\left(\int_0^t  \sigma(Y_t^{(1,0)},Z_t^{(1,0)})\mathrm{d}W_s\right)^2\right)\leq 1. $$
The left hand side above is exactly the expectancy of $Z^{(1,0)}_t-z$ and since $(Z^{(1,0)}_t)_t$ is increasing, we may take the limit in the last inequality to find that $\mathbb{E}(Z^{(1,0)}_{\infty})\leq z+1$. This implies in particular that $Z^{(1,0)}_{\infty}<+\infty$ almost surely. Finally, $(Y_t^{(1,0)},Z^{(1,0)})_t$ converges almost surely when $t\to+\infty$ towards a random variable $(Y^{(1,0)}_{\infty},Z^{(1,0)}_{\infty})$ and Assumption~\ref{ass:main} is verified.

Next, we recall that: \begin{align*}Y_t^{\gamma}&=\sqrt{\gamma}\int_0^t  \sigma(Y_t^{\gamma},Z_t^{\gamma})\mathrm{d}W_s+\int_0^t b(Y_s^{\gamma},Z_s^{\gamma})\mathrm{d}s+y \\
&=\sqrt{\gamma}M_t+B_t, 
\end{align*}
in which $Y_t^{\gamma}$ is bounded in absolute value by $1$ and $B_t$ by $t\times \|b\|_{\infty}+1$; hence Itô's isometry gives us
for all $t\geq 0$ and $\gamma\ge0$: 
$$\sup_{(y,z)}\mathbb{E}_{(y,z)}\left( \gamma \int_0^t \sigma(Y_t^{\gamma},Z_t^{\gamma})^2 \mathrm{d}s\right) \leq (1+t\times \|b\|_{\infty}+1)^2.$$
The term inside the expectancy above is exactly $Z_t^{\gamma}-z$, so using Markov's inequality, we have that for all $\eps>0$, there exists $M_{\eps,t}>0$ such that:

$$\sup_{(y,z)} \mathbb{P}_{(y,z)}\left(|Z_t^{\gamma}-z|>M_{\eps,t}\right)\leq \eps.$$

Taking $K_{\eps,T}=[-1,1]\times [-z_0,z_0+M_{\eps,T}]$ for $z_0>0$ and using that $Z_t^{\gamma}$ is increasing in time for $\gamma$ fixed, we have that:
$$\underset{y\in[-1,1],-z_0\leq z \leq z_0}{\sup}\,\mathbb{P}_{(y,z)}\left(\exists t\in [0,T], (Y_t^{\gamma},Z_t^{\gamma})\in K_{\eps,T}^{c}\right)\leq \eps.$$
The proof works in particular for $b=0$ (when we only consider the dominant process). The compact containment condition in Assumption \ref{ass:tight} is thus checked. 

To prove the remaining assumptions, Assumption \ref{ass:P_cont}, Assumption \ref{ass:non_dirac}, Assumption \ref{ass:effgen} and Assumption \ref{ass:wellposed}, we first need to have an explicit expression of $\Pa$ and its derivative with respect to $y$. 

We start with a time change: the process of generator $\La_{(1)}/\sigma^2$ is of the form:
$$
\begin{cases}
\displaystyle \d \tilde{Y}_t = d W_t \\
\displaystyle  \d \tilde{Z}_t =  d t .
\end{cases}
$$
 and we have $(\tilde{Y}_{S_t},\tilde{Z}_{S_t})=(Y_{t},Z_{t})$ where $\int_0^{t} \sigma^{-2}(Y_s,Z_s)\mathrm{d}s=S_t$. On the one hand if we write $\tau:= \tau_{-1,1}(W)$ the first hitting time of $1$ or $-1$ by the Brownian motion $W$ starting at $y$, we have that $\tilde{Y}_{\tau}\in \{\pm 1\}$ and $\tilde{Y}_t\in ]-1,1[$ for all $t<\tau$. Since $Y_t\in]-1,1[$ for all $t\geq 0$ and $Y_{\infty}=\{\pm 1\}$, we can write $S_{\infty}=\tau$ and obtain $Y_{\infty}=\tilde{Y}_{\tau}=W_{\tau}$ and thus $Z_{\infty}=\tilde{Z}_{\tau}= \tau$. Eventually we have for any $\ph : E_{\eff} \to \R $:
    $$ \Pa(\ph)(y,z) = \mathbb{E}_{(y,z)}(\ph(Y_{\infty},Z_{\infty})) = \E_y[ \ph(W_{\tau}   ,z+  \tau )]$$

As a consequence,
\begin{align*}
 \Pa \varphi(y,z) 
 &=\int_0^{+\infty}\mathbb{P}_y(\tau\in \mathrm{d}t ; W_{\tau}=1)\varphi(1,z+t)\mathrm{d}t \\
 & \ \ \ \ \ \ \ \ +\int_0^{+\infty}\mathbb{P}_y(\tau\in \mathrm{d}t ; W_{\tau}=-1)\varphi(-1,z+t)\mathrm{d}t, 
\end{align*}
where $\mathbb{P}_y(\tau\in \mathrm{d}t ; W_{\tau}=\pm 1)$ is a transparent notation for the probability distribution of $(\tau,W_\tau)$. This explicit expression of $\Pa$ and a dominated convergence argument gives us easily that if $\ph \in C_b(E_{\eff})$, $\Pa \ph$ is continuous on $E$ with limit $\ph$ at $E_\eff$ and hence point Assumption~\ref{ass:P_cont}. The explicit expression of $\Pa$ furthermore gives us immediately Assumption \ref{ass:non_dirac}. Standard regularity results on parabolic PDE (see~\cite[Ch.~$7$]{evans2010partial}) gives us that $(y,z)\mapsto  \Pa \varphi(y,z) $ is in $\mathcal{C}^{\infty}(E)$. For the sake of completeness, we give in appendix an elementary proof of the fact that $\Pa \varphi$ is infinitely differentiable on $E$ using an explicit expression of $\mathbb{P}_y(\tau\in \mathrm{d}t ; W_{\tau}=1)$. We can thus use Itô's formula on to get Assumption~\ref{ass:effgen} point $ii)$.

We denote $\mu,\nu$ the two measures defined for $t>0$ by: $$\left\{\begin{array}{ll}
     \mu(\mathrm{d}t)&=\displaystyle-\underset{y\to1}{\lim}\frac{\partial}{\partial y} \mathbb{P}_y(\tau\in \mathrm{d}t ; W_{\tau}= 1)\mathrm{d}t  \\
     & \\
     \nu(\mathrm{d}t)&=\displaystyle-\underset{y\to1}{\lim}\frac{\partial}{\partial y} \mathbb{P}_y(\tau\in \mathrm{d}t ; W_{\tau}=-1)\mathrm{d}t 
\end{array}\right.$$
and for $t\le0$ by $\mu(\mathrm{dt})=\nu(\mathrm{d}t)=0$. We prove in Appendix~\ref{sec:calc_band} that these two measures are positive, and are in fact Levy measures of subordinators. We are looking for an explicit expression of $\La_{\eff}\varphi(1,z)$ using $\mu$ and $\nu$, but we cannot directly take the limit under the integral when $y$ approaches $1$ in $\partial_y\Pa \varphi$ since $\int \ph(t,z) \mu(\d t)$ could be infinite. We thus need to substract to $\ph(y',z')$ the constant $\ph(1,z)$ and obtain
\begin{align*}
\underset{y\to 1}{\lim}\La_{(0)}\Pa \, \varphi (y,z)
    &= \underset{y\to 1}{\lim}\, b(y,z) \partial_y \Pa \left( \varphi-\varphi(1,z)\right) (y,z) \\ 
\end{align*}

We prove in appendix using a standard dominated convergence argument that we can take the following limit under the integral, $b$ being continuous by assumption:
\begin{align*}
    \underset{y\to 1}{\lim}\La_{(0)}\Pa \, \varphi (y,z)
    &=-b(1,z)\int_{0}^{+\infty} (\varphi(1,z+t)-\varphi(1,z))\mu(\mathrm{d}t) \\ & \ \ \ \ \ \ \ -b(1,z)\int_{0}^{+\infty} (\varphi(-1,z+t)-\varphi(1,z))\nu(\mathrm{d}t) \numberthis{} \label{eq1} \\
\end{align*}
where the last two integrals are well-defined and uniformly bounded in $z$ since we have $\varphi\in \mathcal{C}_c^{\infty}(\{-1,1 \}\times\R)$ and $1\wedge |x|$ integrable for $\mu$ and $\nu$ is finite.
\begin{Rem}
    We remind that $b(1,z)\leq 0$ (the process stays in $E$) and therefore that $-b(1,z)\geq 0$.
\end{Rem}

Doing exactly the same calculus as above, but taking $y\to-1$ instead of $y\to1$, we find:
\begin{align*}
   \underset{y\to -1}{\lim}\La_{(0)}\Pa \, \varphi (y,z) 
    &=b(-1,z)\int_{0}^{+\infty} (\varphi(1,z+t)-\varphi(-1,z))\nu(\mathrm{d}t) \\ & \ \ \ \ \ \ +b(-1,z)\int_{0}^{+\infty} (\varphi(-1,z+t)-\varphi(-1,z))\mu(\mathrm{d}t)\numberthis{} \label{eq2} \\ 
\end{align*}
Combining \eqref{eq1} and \eqref{eq2} and using that $\mu(1\wedge|x|),\nu(1) <+\infty$ we get that $(\pm 1,z)\mapsto\La_{\text{eff}} \, \varphi(\pm 1,z)$ is a bounded measurable function on $E_{\eff}$ and we have therefore that Assumption \ref{ass:effgen} $iii)$ is checked. Furthermore $\La_{\eff}$ is a Levy generator constructed using homogeneous subordinators on $\R_+$: the martingale problem is well-posed and Assumption \ref{ass:wellposed} is checked. A slightly general reference that enables to obtain martingale well-posedness of Levy generators is Theorem $3.1$ of~\cite{kurtz2011}.

\subsection{Diffusion in an Euclidean ball}\label{sec:exa_ball}

We define $E=\overline{\mathbb{B}}^n \subset \R^n$ the closed unit ball of $\R^n$ and $E_{\eff}=\mathbb{S}^{n-1}=\partial \mathbb{B}^n $ its border, the unit sphere. We take  $\calD_{\mrm{eff}}=C_c^\infty\left(\mathbb{S}^{n-1}\right)$ with is the set of all continuous functions on $E_{\eff}$ as space of test functions. We consider a purely diffusive dominant process on the closed unit ball $\overline{\mathbb{B}}^n$ of $\R^n$ and a repulsive drift on the sphere for the subdominant one, we write: $$\mathrm{d}X_t^{\gamma}=\sqrt{\gamma}\sigma(X_t^{\gamma})\mathrm{d}W_t+b(X_t^{\gamma})\mathrm{d}t,$$
where $\sigma: \overline{\mathbb{B}}^n \to \R_+$ is Lipschitz and vanishes \emph{exactly} on $\partial\overline{\mathbb{B}}^n=\mathbb{S}^{n-1}$, while $b:\overline{\mathbb{B}}^n\to \R^n$ is Lipschitz and points towards the interior of the ball at points of $\mathbb{S}^{n-1}$ ($x \cdot b(x) \leq 0$). Our process is thus a strong solution to a SDE with Lipschitz coefficients. Moreover, by extending $\sigma$ and $b$ on $\R^n$ to Lipschitz functions with $\sigma(x) = 0$ and $x\cdot b(x)\leq 0$ for $x$ outside the closed ball, we readily obtain:
$$
\gamma \sigma(x) \Delta_x( \abs{x}^2 -1) + b(x) \cdot \nabla_x( \abs{x}^2 -1) \leq 0, \qquad x \notin \overline{\mathbb{B}}^n
$$
According to Section~\ref{sec:diff}, the process remains in $E$ for all time and we already have Assumption~\ref{ass:flow} and Assumption~\ref{ass:contper}.

We will denote by 
$$b_r(x) \eqdef \langle b(x) , \frac{x}{\abs{x}}  \rangle \in \R , $$
and 
$$
b_\theta(x) \eqdef b(x) - b_r(x)  \frac{x}{\abs{x}} \in \R^n ,
$$
the polar decomposition of $b$. We also denote by 
$$ T\ph(x) \eqdef (\nabla_x - \frac{x}{\abs{x}} \langle \frac{x}{\abs{x}} , \nabla_x \rangle) \ph(x) \in  T_x \S^{n-1},
$$
the (tangent-valued) gradient of a smooth map on the unit sphere $\ph: \S^{n-1} \to \R$ evaluated in $x \in \S^{n-1}$.

\begin{The}
The process $X^{(\gamma,1)}$ with values in $\overline{\mathbb{B}}^n$ converges in law (with respect to the pseudo-path topology and finite-dimensional marginals) when $\gamma$ approaches infinity towards a Levy-type jump process of generator:
    \begin{align*}
    \calL_{\eff} (\ph)(x) =& b_{\theta}(x)\cdot T \ph(x)\\ & \qquad -2b_r(x)\int_{\mathbb{S}^{n-1}} \left(\ph(y)-\ph(x)-\langle T \ph(x), y-x\rangle \right)\mathrm{d}\mu_x(y) .
    \end{align*}
where $\d \mu_x(y))=\frac{1}{\|x-y\|^n}\mathrm{d}\sigma^{n-1}(y)$ is a Levy measure on the sphere with standard metric-induced measure $\sigma^{n-1}$.
\end{The}

We just have to check the six other hypothesis.

Since the dominant process in a bounded martingale it converges almost surely to a point $X_\infty^{(1,0)}$ by Doob's first martingale convergence theorem. By computing its quadratic variation, which must be finite by It\^o isometry, the limit $\sigma(X_\infty^{(1,0)})$ is null and therefore lies in $E_\eff$.  $X_t^{(1,0)}$ converges when $t\to+\infty$ towards $X_\infty^{(1,0)} \in E_\eff$ and Assumption~\ref{ass:main} is checked.

Assumption \ref{ass:tight} is trivially verified since $\overline{\mathbb{B}}^n$ is a compact.

In an Euclidean ball, we do have an explicit expression (see~\cite[Ch.II Sec.~$1$]{doob2012classical}) of the harmonic measure: $$\Pa(x,\cdot)=\left\{\begin{array}{ll}
    \frac{1-\|x\|^2}{\|x-y\|^n} \mathrm{d}\sigma^{n-1}(y)  &\text{if $x\in \mathbb{B}^n$}\\
    \delta_x  &\text{if $x\in \mathbb{S}^{n-1}$}
\end{array}\right.,$$
where $\sigma^{n-1}$ is the Haar probability measure on $\mathbb{S}^{n-1}$. This expression immediately gives us Assumption \ref{ass:non_dirac}.  Since $x\mapsto \Pa(x,\cdot)$ is continuous for the weak convergence topology on $\overline{\mathbb{B}}^n$ (we let the proof in appendix), $x\mapsto \Pa \varphi(x)$ is continuous on $\overline{\mathbb{B}}^n$ at points of $\mathbb{S}^{n-1}$ and we also get Assumption \ref{ass:P_cont}.

Elliptic regularity immediately gives us hat if $\ph$ is smooth on the unit sphere, then $x \mapsto \calP \ph (x)$ is also smooth (see~\cite[Ch.~$6$]{evans2010partial}, \textit{Nota Bene:} this also proves again Assumption~\ref{ass:P_cont}). We may thus use Itô's formula to get Assumption \ref{ass:effgen} point $ii)$.

We now consider the non-radial derivatives, and the explicit expression of $\Pa\varphi$ gives us immediately: $$D_{\theta} \Pa \varphi(x)=\int_{\mathbb{S}^{n-1}} \frac{1-\|x\|^2}{\|x-y\|^n}T\varphi(y)\mathrm{d}\sigma^{n-1}(y).$$
Since $T \varphi$ is in $C^{\infty}(\mathbb{S}^{n-1})$, it is Lipschitz and the continuity of the projector gives us that for all $a\in \mathbb{S}^{n-1}$ we have: \begin{equation}\displaystyle\lim_{z\to a} D_{\theta}\Pa \varphi(z)=T\varphi(a). \label{limrad}\end{equation}

A computation we leave in Appendix~\ref{sec:calc_ball} gives us that:
\begin{equation} \label{exp_Levy}
    \lim_{z\to a}\frac{\partial}{\partial r} \Pa \ph(z)=-2\int_{\mathbb{S}^{n-1}} \frac{1}{\| a-y\|^n}\left(\ph(y)-\ph(a)-\langle T\ph(a), y-a\rangle \right)\mathrm{d}\sigma^{n-1}(y),
\end{equation}
where we notice that, writing $\mu(\mathrm{d}\sigma^{n-1}(y))=\frac{1}{\|a-y\|^n}\mathrm{d}\sigma^{n-1}(y)$, we have: \begin{equation}\int_{\mathbb{S}^{n-1}}(1\wedge\|a-y\|^2) \mathrm{d}\mu(y)<+\infty \label{int_finite}\end{equation}

Thus, combining \eqref{limrad} and \eqref{exp_Levy} and,  we have that:
\begin{align*}\lim_{z\to a}\La_{(0)} \Pa\ph(z) &=-2b_r(a)\int_{\mathbb{S}^{n-1}} \left(\ph(y)-\ph(a)-\langle T\ph(a), y-a\rangle \right)\mathrm{d}\mu(y)\\ & \ \ \ \ \ \ \ \ +b_{\theta}(a)\cdot T\varphi(a).
\end{align*}
Combining the facts that $b$ and all derivatives of $\varphi$ are bounded and \eqref{int_finite}, we have that $\La_{\eff} \varphi$ is measurable bounded on $\mathbb{S}^{n-1}$ and we have point $iii)$ of Assumption \ref{ass:effgen}. Furthermore, the explicit expression of $\La_{\eff}$ gives us that it is the generator of a homogeneous (rotation invariant) Levy process on $\mathbb{S}^{n-1}$: the martingale problem is therefore well-posed and Assumption \ref{ass:wellposed} is checked. It is nonetheless not so easy to find general references that enables to obtain martingale well-posedness of Levy generators on manifold. This can be done here easily with Theorem $3.1$ of~\cite{kurtz2011}, by remarking that on the sphere, a Levy jump from $x$ to $y$, denoted $y-x = \gamma(x,u)$, can be represented by projecting a random direction in $\R^n$ onto the tangent space at $T_x \S^{n-1}$ in order to obtain the random direction of the jump $\gamma(x,u)$. This obviously yields a Lipschitz dependence in $x \mapsto \gamma(x,u)$ and the standard Lipschitz assumption in Theorem~$3.1$ is satisfied.

\newpage
\begin{appendices}
\section{The pseudo-paths topology and the Meyer-Zheng criteria} \label{sec:MeyerZheng}
All results presented in this section can be found in the seminal work of Meyer and Zheng~\cite{MZ84}, up to minor presentation variations.

\subsection{The pseudo-path topology}
Let $E$ be a Polish space. Two measurable paths $t \mapsto x_t \in E$ and $t \mapsto y_t \in E$ belong to the same equivalence class, called a \emph{pseudo-path}, if the set $\set{t: \, x_t =y_t}$ is of full Lebesgue measure (similarly as functions in $\L^p$).  Let us denote by $L^0(\e^{-t} \d t ,E)$ the set of all pseudo-paths. The \emph{pseudo-path topology} is the Polish topology on $L^0(\e^{-t} \d t ,E)$ induced by the closed (see~IV.43~\cite{dellacherie1975} for a proof) injection of $L^0(\e^{-t} \d t ,E)$ in $\Proba(E\times[0,+\infty[)$ endowed with the usual topology of convergence in law. This injection is naturally defined by mapping a pseudo-path $x$ to the probability measure of the form
\begin{equation}\label{eq:pseudo}
\delta_{x_t}(\d x) \e^{-t} \d t .
\end{equation}
In probabilistic terms, $x^n$ converges to $x$ for the pseudo-path topology if and only if the pair $(x^n_T,T)$ converges in distribution towards $(x_T,T)$, where $T$ is exponentially distributed.

By a simple continuity argument, the push-forward~\eqref{eq:pseudo} from paths to pseudo-paths is injective on the set $\D(\R_+,E)$ of c\`adl\`ag paths, so that there also exists a natural injection $\D(\R_+,E) \subset L^0(\e^{-t} \d t ,E)$. It is important to keep in mind that the set of pseudo-paths with a c\`adl\`ag representative is a Borel subset (it is in fact $G_\delta$) but it is \emph{not} closed in the space of pseudo-paths. As a consequence if a sequence of c\`adl\`ag paths converges for the pseudo-paths topology, the limit may or may not be c\`adl\`ag.

There is another natural way to define a Polish topology on the space of pseudo-paths $L^0(\e^{-t} \d t ,E)$ by considering convergence in measure (equivalent here to convergence in probability). This topology can be defined with a probabilistic perspective as follows. A sequence of pseudo-paths $x^n$ converges in probability (or in measure) towards $x$ if and only if
$ x^n_T$ converges in probability towards $x_T$ were $T$ is exponentially distributed. Convergence in probability on $L^0(\e^{-t} \d t ,E)$ is then a Polish topology that can classically by metrized using the $\L^1$-type Ky Fan complete metric (see~\cite[Section~$9.2$]{dudley2018real})
\begin{equation}\label{eq:Ky Fan}
\int_0^\infty \min(d(x_t,y_t),1) \e^{-t} \d t = 0
\end{equation}
where $d$ is complete and metrizing $E$. The following lemma is a slight generalization of Lemma~$1$ in~\cite{MZ84} which is proved for real valued processes. It shows that the above two Polish topologies on pseudo-paths are topologically equivalent (although not metrically).
\begin{Lem}\label{lem:conv_meas}Let $E$ denote a Polish space. On pseudo-paths space $L^0(\e^{-t} \d t ,E)$, the Polish topology of convergence in probability induced by~\eqref{eq:Ky Fan} is equivalent to the Polish topology of convergence in distribution induced by~\eqref{eq:pseudo}. 
\end{Lem}
\begin{proof} Let us prove that convergence in probability implies convergence as distributions on $\R_+\times E$. Let $T$ be an exponentially distributed random variable. By definition, if $x^n_T$ converges in probability towards $x_T$ so is the pair $(x^N_T,T)$ towards $(x_T,T)$. This implies convergence in distribution in $\Proba(E\times[0,+\infty[)$.

Conversely, let us prove that convergence $x^n \to x$ as  distributions on $\R_+\times E$ implies convergence in probability. The following argument is a minor adaptation of the proof of Lemma~$1$ of~\cite{MZ84} which treats the case $E = \R$.

Classically, we can (homeomorphically) embed $E$ in the unit ball of a separable Hilbert space. For instance, $E$ can be embed in the Hilbert cube $\phi : E \to [0,1]^{\N}$ by considering the countable family of bounded continuous test functions $\phi_k(x) = d(x,x_k)$ where $(x_k)_{k \in \N}$ is a dense subset of $E$, and $d \leq 1$ is a bounded distance metrizing $E$; indeed, a sequence converges in $E$ towards a given limit in $E$ if and only if its images by $\phi_k$ converges for each $k \geq 0$. As a consequence, the weighted $\ell^2_w$ Hilbert space defined by the norm
\[
\widetilde d(x,y)^2  \eqdef \norm{\phi(x) - \phi(y)}^2_{\ell^2_w} \eqdef \sum_{k \geq 0} 2^{-k-1} \abs{\phi_k(x)-\phi_k(y)}^2
\]
also metrizes the Polish space $E$ although $\widetilde d$ may not be complete.

We can next consider the Hilbert space $L^2(\e^{-t} \d t,\ell^2_w)$ of measurable paths on $\R_+$ taking values in the separable Hilbert $\ell^2_w$. By a routine finite dimensional approximation argument (one can consider increasing finite-dimensional sub-spaces $\ell_d \subset \ell^2_w$ with $l_d \nearrow_d \ell^2_w$), the space of bounded continuous functions from $\R_+$ to $\ell^2_w$ is dense in $L^2(\e^{-t} \d t,\ell^2_w)$. As a consequence, a converging sequence of pseudo-paths $x^n \to x$ (as distributions of $\Proba(E\times[0,+\infty[)$) also converges weakly (by definition, and since $x^n$ and $x$ lie in the unit ball of $\ell^2_w$) in $L^2(\e^{-t} \d t,\ell^2_w)$, as can be seen by considering the convergence of the scalar product $\int_0^\infty \langle \phi(x^n_t) , \ph(t)  \rangle \e^{-t} \d t $ for any continuous and bounded path $\ph$ with values in $\ell^2_w$. In the same way, convergence as pseudo-paths implies that the Hilbert norm in $L^2(\e^{-t} \d t,\ell^2_w)$ of $x^n$ converges to the one of $x$. Hence, $x^n$ converges strongly to $x$ in $L^2(\e^{-t} \d t,\ell^2_w)$. Since the latter two lie in the unit ball, strong convergence in $L^2(\e^{-t} \d t,\ell^2_w)$ is equivalent to convergence in probability in $\ell^2_w$ or equivalently in $E$ since convergence in probability does not depend on the metric.
\end{proof}

Now, one can considers \emph{random} pseudo-paths defined as classes of equivalence of measurable maps $X:(\Omega,\P) \to L^0(\e^{-t} \d t,E)$ identified with almost sure equality. It is thus possible consider the set of random pseudo-paths as the space $L^0(\P,L^0(\e^{-t} \d t,E))$. The following lemma is somehow trivial, but nonetheless clarifying.
\begin{Lem} The metric
$$
\E \int_0^\infty \min(d(X_t,Y_t),1) \e^{-t} \d t
$$
is an isometry between $L^0(\P,L^0(\e^{-t} \d t,E))$ and $L^0(\P \otimes \e^{-t} \d t,E)$, when we metrize $L^0(\e^{-t} \d t,E)$ with~\eqref{eq:Ky Fan}. In particular, any random pseudo-path can be represented by a $\R_+\times \Omega$-measurable process $(t,\omega) \mapsto X(t,\omega)$.
\end{Lem}
\begin{proof}
Since $E$ is separable, there exists $C \subset E$ a countable dense subset of $E$. Random variables taking values in $C$ are dense for convergence in probability. The isometry is trivial for such random variables by considering all maps $(t,\omega) \mapsto X(t,\omega) \in C$ which are automatically measurable.
\end{proof}

\subsection{Convergence of finite-dimensional distributions}
Convergence in pseudo-path space of c\`adl\`ag processes does not imply convergence of finite dimensional distribution (as is already the case for the Skorokhod topology). However, up to extraction, one can obtain convergence of finite dimensional distributions for all times in a set of full Lebesgue measure. This was proved as Theorem~$5$ in~\cite{MZ84} for general real values processes. We sketch the proof again in the specific c\`adl\`ag case but for Polish-valued processes. 

\begin{Lem}\label{lem:finiteMZ} Let $E$ be Polish and let $X^n$ be a sequence of c\`adl\`ag processes converging in distribution,  for the pseudo-path topology, towards a c\`adl\`ag $X^\infty$ process. Then there exists a sub-sequence and a subset  $J \subset \R_+$ of full Lebesgue measure for which all finite-dimensional distributions converge.
\end{Lem}
\begin{proof} By the Skorokhod representation theorem in Polish spaces, one can assume that $\widetilde X^n$ converges towards $\widetilde X^\infty$ almost surely, where $\widetilde X^n$ has the same pseudo-path distribution as $X_n$, for each $n\in \llbracket 1,+\infty\rrbracket$. Since by assumption $\P(\widetilde X_n \in \D(\R_+,E) \subset L^0(\e^{-t} \d t ,E))=1$, we can canonically consider $\widetilde X^n$ as a c\`adl\`ag process using its unique c\`adl\`ag representative. We thus subsequently drop the $\,\,  \widetilde {} \,\,$ notation.

By Lemma~\ref{lem:conv_meas}, almost sure convergence in pseudo-paths space implies almost sure convergence in probability. Taking the expectation and using dominated convergence one gets
\[
\int_0^\infty  \E \b{ \min(d( X^n_t, X^\infty_t),1) } \e^{-t} \d t \xrightarrow[n \to \infty]{} 0,
\]
that is convergence in probability but with respect to the product measure $\e^{-t} \d t \otimes \P$. As a consequence there exists a sub-sequence (we do not change notation to denote this sub-sequence) and a set $J \times \Omega$ of full $\e^{-t} \d t \otimes \P $ measure on which $ X^{n}$ converges towards $ X^{\infty}$. This implies convergence of finite dimensional distributions for all times in $J$.
\end{proof}

\subsection{Compact sub-spaces}
Compact subsets of the pseudo-paths space can be easily characterized using a direct application of  the Prokhorov theorem.
\begin{Lem} A subset $\calK$ of pseudo-paths is compact if and only if for any $\eps > 0$, there exists $t_\eps >0$ and $K_\eps \subset E$ compact such that
$$
\inf_{x \in \calK} \int_0^{t_\eps} \one_{d(x_t,K_\eps) \leq \eps} \, \e^{-t} \d t \geq 1 - \eps ;
$$
in particular, if for all $T >0$, all the pseudo-paths of $\calK$ are taking on $[0,T]$ their values in a compact $K_T \subset E$, then the pseudo-paths space $\calK$ is compact.
\end{Lem}

\subsection{The Meyer-Zheng criterion}
We now turn to the case where $E=\R$. The idea is to consider the concept of \emph{mean variation} on pseudo-path distributions. A c\`adl\`ag random process with finite mean variations characterizes (integrable) c\`adl\`ag quasi-martingales. 
\begin{Def}[Mean variation] Let $T > 0$ denotes an horizon time, and $(X_t)_{t \geq 0}$ a c\`adl\`ag real valued random process with $Z_t \in L^1$ for all $t$. The mean variation of $Z$ over $[0,T]$ with respect to its own natural filtration is given by
\begin{align*}
V_T(Z) \eqdef & \sup_{0=t_0\leq t_1 \ldots \leq t_K \leq T} \sum_{k=0}^K \E \abs{\E\b{ Z_{t_{k+1}}-Z_{t_k} \mid \sigma\p{ Z_{t}, \, t \leq t_k }} }, \\
= & \underset{\abs{H}\leq 1}{\sup}\mathbb{E} \left( \int_0^{T} H_t \, \d Z_t \right),
\end{align*}
where in the above the supremum is taken over predictable processes (with respect to the natural filtration of $Z$) taking value in $[-1,1]$.
\end{Def}
As noted in Remark~$1$ p.~$362$ of~\cite{MZ84} the mean variations is increasing with the considered filtration, so that proving a uniform bound for a larger filtration immediately yields the same bound for the natural filtration of the considered process.

The main theorem of Meyer-Zheng (Theorem~$4$) can be presented as follows.

\begin{The}[Meyer-Zheng cirteria] \label{the:cadlagMZ}
Let $(\pi_n)_{n \geq 1}$ denotes a sequence of distributions of real valued c\`adl\`ag processes. Assume that this sequence is converging for the pseudo-path topology towards a distribution in pseudo-path space $\pi$. If for all $T >0$ the mean variation is uniformly bounded  $\limsup_n V_T(\pi_n) < +\infty$, then the pseudo-path limit is c\`adl\`ag, that is $\pi_\infty$  fully charges $\D(\R_+,\R) \subset L^0(\e^{-t} \d t  ,\R)$, and it has bounded mean variation: $V_T(\pi) \leq \limsup_n V_T(\pi_n) < +\infty$ for all $T$.
\end{The}
\begin{proof}
This is a direct adaptation of Theorem $4$ in~\cite{MZ84}.

Let $T>0$ by given and consider the distribution $(\pi^T_n)_{n \geq 1}$ and $\pi^T$ of pseudo-paths stopped at $T$ (equal to their value at $T$ for all $t \geq T$). The sequence $(\pi^T_n)_{n \geq 1}$ again converges towards $\pi^T$. Using the proof of Theorem $4$ in~\cite{MZ84} after the preliminary extraction step (which is superfluous here), we get that $\pi^T$ has support in $\D$. Using Remark~$2$ after the proof, we also obtain that $V_T(\pi) \leq \limsup_n V_T(\pi_n) < +\infty$.

Finally, one can naively remark that a pseudo-path $x$ has a c\`adl\`ag representative if and only if the stopped pseudo-path $x^T$ has a c\`adl\`ag representative for all integer $T$:
\[
x \in \D(\R_+,E) \subset L^0(\e^{-t} \d t ,E) \Leftrightarrow x^T \in \D(\R_+,E)\subset L^0(\e^{-t} \d t ,E), \, \, \forall T \in \N,
\]
by identifying the unique c\`adl\`ag representative on each interval $[T,T+1]$. By $\sigma$-additivity, it follows that $\pi$ fully charges $\D$.
\end{proof}


\subsection{Characterization of c\`adl\`ag martingales}

We will also use in this article the following characterization of the martingale property on c\`adl\`ag processes: 
\begin{The} \label{martingale_criteria}
A bounded c\`adl\`ag process $\mathcal{M}$ is a martingale for the filtration of $(\mathcal{F}_t)_{t\geq 0}$ if and only if:
$$\mathbb{E}[ (\mathcal{M}_t-\mathcal{M}_{t_k})\varphi_k(X_{t_k}) ... \varphi_1(X_{t_1}) ] =  0$$
for all $t_1 < ... < t_k < t$ in a dense subset, $\varphi_1, ... , \varphi_k$ continuous and bounded and for all $k$.
\end{The}

It is then possible to use this property to prove that the limit of a sequence of martingales is again a martingale. Typically, the considered sequence of martingales is constructed from a sequence of c\`adl\`ag Markov processes that converge to a c\`adl\`ag process \emph{for the pseudo-path topology}; and the convergence of almost all finite dimensional distributions is sufficient   to a martingale (closure property). This will be carried out in Step~$3$ of Section~\ref{sec:proofs} while proving the main theorem.

\section{Calculus associated with the diffusion in a strip}\label{sec:calc_band}

\subsection{Regularity of $\Pa \varphi$}
$\Pa \varphi$ is given by the following explicit expressions of $\mathbb{P}_y(\tau\in \mathrm{d}t ; W_{\tau}= 1 )$ and $\mathbb{P}_y(\tau\in \mathrm{d}t ; W_{\tau}= 1 )$ based on theta functions (\cite[Section~$3.0.6$ and Section~$11$, App.~$2$, p.641]{borodin2015handbook}):

\begin{equation} \label{expansian}
 \mathbb{P}_y(\tau\in \mathrm{d}t ; W_{\tau}= 1 ) = \sum_{k= - \infty}^{+\infty} \frac{4k+1 - y}{\sqrt{2\pi} t^{3/2}} \e^{ - (4k+1 - y)^2/(2t)},
\end{equation}
and:
\begin{equation}\label{expansian2} \mathbb{P}_y(\tau\in \mathrm{d}t ; W_{\tau}=-1 ) = \sum_{k= - \infty}^{+\infty} \frac{4k+1+y}{\sqrt{2\pi} t^{3/2}} \e^{ - (4k+1+y)^2/(2t)} . \end{equation}

For the sake of simplicity, we will mainly consider \eqref{expansian}, the computations with \eqref{expansian2} being identical by symmetry. We claim that if $\ph \in C_c^\infty(\R)$, then $y \mapsto \int_{0}^\infty \ph(t) \mathbb{P}_y(\tau\in \mathrm{d}t ; W_{\tau}= 1 )$ is smooth on $[-1,1]$. The terms in~\eqref{expansian} for $k \neq 0$ are easy to handle by routine dominated convergence. For $k=0$, we remark that a difficulty arises for $y=1$, but that the associated term is of the form $g(1-y)$ where $g$ is the following smooth function:
\begin{Lem} Let $\ph$ be smooth with compact support. The function on $\R_+$ defined by 
$$
g(u) \eqdef \int_0^{\infty} \e^{-u^2/t} u \, t^{-\frac{3}{2}} \ph(t) \d t
$$
is smooth on $\R_+$.
\end{Lem}
\begin{proof} By a change of scale $u^2/t = 1/t' $:
\begin{align*}
g(u) &= \int_0^{\infty} \e^{-1/t'} u (u^2t')^{-\frac{3}{2}} \ph(u^2 \, t') u^2\d t' \\
&=\int_0^{\infty} \e^{-1/t'} u^{1-3+2} t'^{-\frac{3}{2}} \ph(u^2 \, t') \d t' \\
&=\int_0^{\infty} \e^{-1/t'} t'^{-\frac{3}{2}} \ph(u^2 \, t') \d t'.
\end{align*}
Using that the support of $\ph$ is compact, we have that the function is continuous on $\R_+$ and continuously differentiable on $\R_+^*$. We have after a change of scale $u^{2}/s=1/t'$:
\begin{align*}
g'(u)&=\int_0^{\infty} \e^{-1/t'} t'^{-\frac{3}{2}} \ph'(u^2 \, t')2 t' u \d t' \\
&=2\int_0^{\infty} \e^{-u^2/s} \left(\frac{u^2}{s}\right)^{\frac{1}{2}} \ph'(s) u \frac{1}{u^2} \d s \\
&=2\int_0^{\infty} \e^{-u^2/s} s^{-\frac{1}{2}} \ph'(s)\d s \\
&\underset{u\to 0^+}{\longrightarrow} 2\int_{0}^{+\infty} s^{-\frac{1}{2}}\ph'(s)\d s.
\end{align*}
We have hence that $g$ is continuously differentiable on $\R_+$. Deriving a second time we find:
\begin{align*}
    g''(u)&=2\int_{0}^{+\infty} e^{-\frac{u^2}{s}}\left(-\frac{2u}{s}\right)s^{-\frac{1}{2}}\ph'(s)\mathrm{d}s \\
    &=-4\int_{0}^{+\infty} e^{-\frac{u^2}{s}} u \,s^{-\frac{3}{2}} \ph'(s)\mathrm{d}s
\end{align*}
The form of this function is exactly the one of $g$ up to multiply by factor $-4$ and replace $\ph$ by $\ph'$. Since we prove that $g$ is continuously differentiable on $\R$, we have by an immediate induction that $g$ is in fact infinitely differentiable on $\R_+$.
\end{proof}
We therefore have that $y\mapsto \Pa(y,z)$ is smooth on $[-1,1]$ for all $z\in \R$. Since $\Pa \ph$ is trivially smooth in $z$ by regularity of $\varphi$, we have that $(y,z)\mapsto \Pa \varphi(y,z)$ is smooth on $E=[-1,1] \times \R$.

\subsection{The measures $\mu$ (resp. $\nu$) are positive Levy of a subordinator (resp. positive finite)}

In this section we will study $\mu$ and $\nu$ and then prove that we can take the limit under the integral when $y$ approaches $\pm 1$ in $\partial_y \Pa \varphi(y,z)$. In order to do that, we will first give a simple expression of $\partial_y \Pa \varphi(y,z)$:

\begin{Lem} For all $y\in ]-1,1[$ we have:
\begin{equation}\partial_y \mathbb{P}_y(\tau\in \mathrm{d}t,W_{\tau}=1)=-\frac{\pi^2}{16}\sum_{k=-\infty}^{+\infty} \cos \left(k\pi \frac{1-y}{2}\right) k^2 e^{-t\frac{\pi^2 k^2}{8}}, \numberthis \label{der_s1}\end{equation}
and:
\begin{equation}\partial_y \mathbb{P}_y(\tau\in \mathrm{d}t,W_{\tau}=-1)=\frac{\pi^2}{16}\sum_{k=-\infty}^{+\infty} \cos \left(k\pi \frac{1+y}{2}\right) k^2 e^{-t\frac{\pi^2 k^2}{8}}. \numberthis \label{der_s2}\end{equation}
\end{Lem}
\begin{proof}
    We start with the first formula. Deriving \eqref{expansian} under $y$ once we obtain:
\begin{align*}
    &\frac{\partial}{\partial y} \mathbb{P}_y(\tau\in \mathrm{d}t, W_{\tau}=1) \\ &=\sum_{k=-\infty}^{+\infty} -\frac{1}{\sqrt{2\pi} t^{\frac{3}{2}}}\e^{ -\frac{ (4k+1 - y)^2}{2t}}+\frac{4k+1-y}{\sqrt{2\pi} t^{\frac{3}{2}}}\e^{ -\frac{ (4k+1 - y)^2}{2t}}\times \frac{2(4k+1-y)}{2t} \\
    &= \frac{1}{\sqrt{2\pi} t^{\frac{3}{2}}}\sum_{k=-\infty}^{+\infty} e^{-\frac{(4k+1-y)^2}{2t}} \left(-1+ \frac{(4k+1-y)^2}{t} \right) \numberthis{} \label{der}
\end{align*}
    
    Writing $f:t\mapsto e^{-\pi a (t+\frac{1-y}{4})^2}$ and $\mathcal{F} f: t\mapsto \displaystyle\int_{\R} e^{-2i\pi t u} f(u)\mathrm{d}u$, Poisson summation formula gives us:
$$\sum_{k=-\infty}^{+\infty} f(k)=\sum_{k=-\infty}^{+\infty} \mathcal{F} f(k)$$
But, for $g:t\mapsto e^{-\pi a t^2}$, we have $(\tau_{-\frac{1-y}{4}}) \ast g=f$ and $\mathcal{F}g(t)=a^{-\frac{1}{2}}e^{-\frac{\pi}{a}t^2}$, thus:
\begin{align*}
    \mathcal{F} f(k)&=\mathcal{F}((\tau_{-\frac{1-y}{4}})\ast g)(k) \\
    &=e^{ik\pi\frac{1-y}{2}}\mathcal{F} g (k) \\
    &=e^{ik\pi \frac{1-y}{2}} a^{-\frac{1}{2}}e^{-\frac{\pi}{a}k^2}
\end{align*}

Eventually, we have:
\begin{align*}
    \sum_{k=-\infty}^{+\infty} e^{-\pi a (k+\frac{1-y}{4})^{2}}&=\sum_{k=-\infty}^{+\infty} e^{ik\pi\frac{1-y}{2}} a^{-\frac{1}{2}}e^{-\frac{\pi}{a}k^2} \\
    &=\frac{1}{2}\sum_{k=-\infty}^{+\infty} 2\cos\left(k\pi \frac{1-y}{2}\right) a^{-\frac{1}{2}}e^{-\frac{\pi}{a}k^2} \\
    &=\sum_{k=-\infty}^{+\infty} \cos\left(k\pi \frac{1-y}{2}\right) a^{-\frac{1}{2}}e^{-\frac{\pi}{a}k^2}
\end{align*}
If we differentiate this equation under $a$ we get:
\begin{align*}
    &\sum_{k=-\infty}^{+\infty} -\pi \left(k+\frac{1-y}{4} \right)^2 e^{-\pi a (k+\frac{1-y}{4})^2} \\  &=-\frac{1}{2}a^{-\frac{3}{2}}\sum_{k=-\infty}^{+\inf'"ty} \cos\left(k\pi \frac{1-y}{2}\right) e^{-\frac{\pi}{a}k^2}+a^{-\frac{1}{2}}\sum_{k=-\infty}^{+\infty} \cos\left(k\pi \frac{1-y}{2}\right) e^{-\frac{\pi}{a}k^2} \frac{\pi k^2}{a^2} 
\end{align*}

Taking $a=\frac{16}{2\pi t}$, and we have respectively:
\begin{equation} \label{sum_lem}\sum_{k=-\infty}^{+\infty} e^{- \frac{(4k+1-y)^{2}}{2t}}=\sum_{k=-\infty}^{+\infty} \cos\left(k\pi \frac{1-y}{2}\right) \frac{\sqrt{2\pi t}}{4}e^{-t\frac{\pi^2 k^2}{8}}, \end{equation}
and:
\begin{align*}
    &\sum_{k=-\infty}^{+\infty} -\pi \left(k+\frac{1-y}{4} \right)^2 e^{-\frac{ (4k+1-y)^2}{2t}} \\  &=-\frac{1}{2}\frac{\sqrt{2\pi t}^3}{64}\sum_{k=-\infty}^{+\infty} \cos\left(k\pi \frac{1-y}{2}\right) e^{-t\frac{\pi^2 k^2}{8}} \\ & \ \ \ \ \ \ \ \ \ +\frac{\sqrt{2\pi t}}{4}\sum_{k=-\infty}^{+\infty} \cos\left(k\pi \frac{1-y}{2}\right) e^{-t\frac{\pi^2 k^2}{8}} \frac{\pi^3 t^2 k^2}{64} \numberthis \label{der_lem}
\end{align*}
We now multiply \eqref{sum_lem} by $-\frac{1}{\sqrt{2\pi}t^{\frac{3}{2}}}$ and \eqref{der_lem} by $-\frac{1}{\sqrt{2\pi}t^{\frac{3}{2}}}\times \frac{16}{\pi t}$ to get:
\begin{align*}
    &\frac{1}{\sqrt{2\pi} t^{\frac{3}{2}}}\sum_{k=-\infty}^{+\infty} e^{-\frac{(4k+1-y)^2}{2t}} \left(-1+ \frac{(4k+1-y)^2}{t} \right) \\
    &= -\frac{1}{4t}\sum_{k=-\infty}^{+\infty} \cos\left(k\pi \frac{1-y}{2}\right) e^{-t\frac{\pi^2 k^2}{8}}+\frac{1}{4t} \sum_{k=-\infty}^{+\infty} \cos\left(k\pi \frac{1-y}{2}\right) e^{-t\frac{\pi^2 k^2}{8}} \\
    & \ \ \ \ \ \ \ -\frac{\pi^2}{16} \sum_{k=-\infty}^{+\infty} \cos\left(k\pi \frac{1-y}{2}\right) k^2 e^{-t\frac{\pi^2 k^2}{8}} \\
    &=-\frac{\pi^2}{16} \sum_{k=-\infty}^{+\infty} \cos\left(k\pi \frac{1-y}{2}\right) k^2 e^{-t\frac{\pi^2 k^2}{8}}
\end{align*}
Combining this with \eqref{der}, we get exactly \eqref{der_s1}. The proof of the second formula follows exactly the same steps, except we consider $\tau_{-\frac{1+y}{4}}$ instead of $\tau_{-\frac{1-y}{4}}$.
\end{proof}

Let us take the limit in \eqref{der_s1} when $y$ approaches $1$:
\begin{equation}\label{principal}
   \lim_{y\to 1} \frac{\partial}{\partial y} \mathbb{P}_y(\tau\in \mathrm{d}t, W_{\tau}=1)=-\frac{\pi^2}{16} \sum_{k=-\infty}^{+\infty} k^2 e^{-t\frac{\pi^2 k^2}{8}}
\end{equation}
and define $\mu(\mathrm{d}t) \eqdef \left(\frac{\pi^2}{16} \displaystyle\sum_{k=-\infty}^{+\infty} k^2 e^{-t\frac{\pi^2 k^2}{8}} \right) \mathrm{d}t$ if $t> 0$, while $\mu(\mathrm{d}t)= 0$ if $t \leq 0$. This coincides with the definition of $\mu$ given in the text and we have that $\mu$ is trivially a positive measure.

Let us show that $\mu$ is the Levy measure of a subordinator. We have that $\mu\geq 0$ and:
\begin{align*}
    &\int_0^{+\infty} \min(1,t)\mu(\mathrm{d}t)\\&=\int_{0}^1 t \left(\frac{\pi^2}{16}\sum_{k=-\infty}^{+\infty} k^2 e^{-t\frac{\pi^2 k^2}{8}}\right)\mathrm{d}t+\int_{1}^{+\infty} \left(\frac{\pi^2}{16}\sum_{k=-\infty}^{+\infty} k^2 e^{-t\frac{\pi^2 k^2}{8}}\right)\mathrm{d}t\\
    &=\frac{\pi^2}{16}\sum_{k=-\infty}^{+\infty} k^2\int_{0}^1 t e^{-t\frac{\pi^2 k^2}{8}}\mathrm{d}t+\frac{\pi^2}{16}\sum_{k=-\infty}^{+\infty} k^2\int_{1}^{+\infty}e^{-t\frac{\pi^2 k^2}{8}} \mathrm{d}t \\
    &=\frac{\pi^2}{16}\sum_{k=-\infty}^{+\infty} k^2\left(\left[-\frac{8}{\pi^2 k^2}te^{-t\frac{\pi^2 k^2}{8}}\right]_0^1+\frac{8}{\pi^2 k^2}\int_{0}^1 e^{-t\frac{\pi^2 k^2}{8}} \mathrm{d}t \right) \\
    & \ \ \ \ \ \ \ \ +\frac{\pi^2}{16}\sum_{k=-\infty}^{+\infty} k^2 \left[-\frac{8}{\pi^2 k^2} e^{-t\frac{\pi^2 k^2}{8}} \right]_1^{+\infty} \\
    &=\frac{\pi^2}{16}\sum_{k=-\infty}^{+\infty} k^2\left(-\frac{8}{\pi^2 k^2}e^{-\frac{\pi^2 k^2}{8}}+\left(\frac{8}{\pi^2 k^2}\right)^2(1-e^{-\frac{\pi^2 k^2}{8}}) \right) \\
    & \ \ \ \ \ \ \ \ +\frac{\pi^2}{16}\sum_{k=-\infty}^{+\infty} k^2 \frac{8}{\pi^2 k^2} e^{-\frac{\pi^2 k^2}{8}} \\
    &<+\infty \numberthis\label{cvinf}
\end{align*}

Let us now take the limit in \eqref{der_s2} when $y$ approaches $1$:

\begin{equation}\label{principal_2}
   \lim_{y\to 1} \frac{\partial}{\partial y} \mathbb{P}_y(\tau\in \mathrm{d}t, W_{\tau}=1)=-\frac{\pi^2}{16} \sum_{k=-\infty}^{+\infty} (-1)^k k^2 e^{-t\frac{\pi^2 k^2}{8}},
\end{equation}
and define $\nu(\mathrm{d}t)\eqdef \left(-\frac{\pi^2}{16} \displaystyle\sum_{k=-\infty}^{+\infty} (-1)^k k^2 e^{-t\frac{\pi^2 k^2}{8}} \right) \mathrm{d}t$ if $t>0$, while $\nu(\mathrm{d}t)=0$ if $t \le0$. This coincides with the definition of $\nu$ given in the text and we will now prove that this measure is positive and finite of mass
$$
\int_{- \infty} ^ \infty \nu(\d t) = 1/2 .
$$

To show that $\displaystyle\sum_{k=-\infty}^{+\infty} (-1)^k k^2 e^{-t\frac{\pi^2 k^2}{8}}\leq 0$ for all $t>0$, we will consider two cases.

We first assume that $t\geq 1$. Then the sequence $\left(k^2e^{-t\frac{\pi^2 k^2}{8}}\right)_{k\geq 1}$ is decreasing since the function $x\mapsto x^2 e^{-t\frac{\pi^2 x^2}{8}}$ has a negative derivative for $x^2\geq \frac{8}{t\pi^2}$ and thus for $x\geq 1$ in our situation. Hence for $t\geq 1$ we have:
\begin{equation*}
    \sum_{k=-\infty}^{+\infty} (-1)^k k^2 e^{-t\frac{\pi^2 k^2}{8}}=2\sum_{k=1}^{+\infty} (-1)^k k^2 e^{-t\frac{\pi^2 k^2}{8}}\leq 0
\end{equation*}
by the criteria of convergence of alternate series.
\\

We now write:
\begin{align}
    \sum_{k=-\infty}^{+\infty} (-1)^{k} k^2 e^{-t\frac{\pi^2 k^2}{8}}&=-\frac{8}{\pi^2}\frac{\partial}{\partial t}\left[ \sum_{k=-\infty}^{+\infty} (-1)^k e^{-t\frac{\pi^2 k^2}{8}} \right] \nonumber \\
    &=-\frac{8}{\pi^2}\frac{\partial}{\partial t} \left[ \sum_{k=-\infty}^{+\infty}  e^{-i\pi k-t\frac{\pi^2 k^2}{8}}\right].\label{eq:int_nu}
\end{align}
We will next prove, in order to conclude on the positivity of $\nu$, that $t\mapsto \displaystyle \sum_{k=-\infty}^{+\infty}  e^{-i\pi k-t\frac{\pi^2 k^2}{8}}$ is increasing for $t\in ]0,1[$. Writing $h:x\mapsto e^{-t\frac{\pi^2 x^2}{8}}$, we have that: $$\mathcal{F}(e^{-i\pi x} h)(\xi)=(\tau_{-\frac{1}{2}}) \ast \mathcal{F}(h)(\xi)=\sqrt{\frac{8}{\pi t}}e^{-\frac{8(\xi+\frac{1}{2})^2}{t}}$$
Thus, the Poisson summation formula gives:
\begin{align}
F(t) \eqdef \sum_{k=-\infty}^{+\infty}  e^{-i\pi k-t\frac{\pi^2 k^2}{8}}&=\sum_{k=-\infty}^{+\infty} e^{-i\pi k}h(k)\\
&=\sum_{k=-\infty}^{+\infty}  \mathcal{F}(e^{-i\pi x}h)(k) \\
&=\sqrt{\frac{8}{\pi t}}\sum_{k=-\infty}^{+\infty} e^{-\frac{8(k+\frac{1}{2})^2}{t}} \label{eq:int_poiss_nu}
\end{align}
We consider for any $k\in \Z$, the function $f_k:t\mapsto t^{-\frac{1}{2}} e^{-\frac{8(k+\frac{1}{2})^2}{t}}$, we have that $f_k':t\mapsto e^{-\frac{8(k+\frac{1}{2})^2}{t}}t^{-\frac{3}{2}}\left[-\frac{1}{2}+\frac{8(k+\frac{1}{2})^2}{t}\right]$. Thus, $f_k'(t)\geq 0$ if and only if $-\frac{1}{2}+\frac{8(k+\frac{1}{2})^2}{t}\geq 0$, that is to say $t\leq 16(k+\frac{1}{2})^2$, but this inequality is true for every $k$ if we have that $t\leq 4$. Since we assume that $t\in]0,1[$, we have that 
$F$ is increasing for $t\in]0,1[$.

Hence one obtains finally, 
for all $t>0$, $-\frac{\pi^2}{16} \displaystyle\sum_{k=-\infty}^{+\infty} (-1)^k k^2 e^{-t\frac{\pi^2 k^2}{8}}\geq 0$, and thus $\nu\geq 0$.

To compute the total mass of $\nu$, consider~\eqref{eq:int_nu} to get
\begin{align*}
    \int_0^{+\infty} \nu(\mathrm{d}t) = \frac12( F(+\infty) - F(0)),
\end{align*}
and remark that $F(+\infty) = 1$, while by~\eqref{eq:int_poiss_nu} $F(0) = 0$.
\begin{Rem}
    We also have that $\mu(\mathrm{d}t)=\underset{y\to -1}{\lim} \partial_y \mathbb{P}_y(\tau\in \mathrm{d}t, W_{\tau}=-1) \mathrm{d}t$ and $\nu(\mathrm{d}t)=\underset{y\to 1}{\lim} \partial_y \mathbb{P}_y(\tau\in \mathrm{d}t, W_{\tau}=1)\mathrm{d}t$.
\end{Rem}

We proved the wanted properties on $\mu$ and $\nu$, however we still have to show that we can take the limit under the integral in the explicit expression of $\partial_y \Pa (\varphi-\varphi(1,z))(y,z)$ to get \eqref{eq1}. The key idea is to remark that, since $|\cos(k\pi\frac{1-y}{2})|\leq 1$ and $|\cos(k\pi \frac{1+y}{2})|\leq 1$, equations \eqref{der_s1} and \eqref{der_s2} give us that for all $y\in ]-1,1[$ and $t>0$:
$$\max\left(\left|\partial_y \mathbb{P}_y(\tau\in \mathrm{d}t,W_{\tau}=1)\right|,\left|\partial_y \mathbb{P}_y(\tau\in \mathrm{d}t,W_{\tau}=-1)\right| \right)\leq \frac{\pi^2}{16}\sum_{k=-\infty}^{\infty}k^2 e^{-t\frac{\pi^2 k^2}{8}}.$$
Indeed, we proved in \eqref{cvinf} that this quantity is integrable on $\R_{+}$ against $\min(1,t)$, and $\varphi(\pm 1,z+\cdot)-\varphi(\pm 1,z)$ is bounded by a constant time $\min(1,t)$, so by the dominated convergence theorem we can take the limit when $y$ approaches $\pm 1$ to get what we want.

\section{Calculus associated with the diffusion in the Euclidean ball}\label{sec:calc_ball}
Elliptic regularity immediately gives us hat if $\ph$ is smooth on the unit sphere, then $x \mapsto \calP \ph (x)$ is also smooth (see~\cite[Ch.~$6$]{evans2010partial}). For the sake of completeness we nonetheless give explicit calculations that are the most relevant to our context.
\subsection{Continuity of the projector}
Let us consider $x_k\in \overline{\mathbb{B}}^n$ that converges to $x$. The explicit formula for $\Pa$ implies that $\Pa(x_k,\cdot) \to_k \Pa(x,\cdot)$ if either $x\in \mathbb{B}^n$ or if $x_k\in \mathbb{S}^{n-1}$ for all $k$, we can thus assume without lost of generality that $x_k\in \overline{\mathbb{B}}^n$ and $x\in \mathbb{S}^{n-1}$. Then we have for all $\alpha>0$ and $n$ such that $\|x-x_k\|\leq \frac{\alpha}{2}$: \begin{align*}\int_{\mathbb{S}^{n-1}\setminus B(x,\alpha)}\frac{1-\|x_k\|^2}{\|x_k-y\|^n} \mathrm{d}\sigma^{n-1}(y)&\leq (1-\|x_k\|^2)\times \underset{y\in \mathbb{S}^{n-1}\setminus B(x,\alpha)}{\sup} \|x_k-y\|^{-n} \\
&\leq (1-\|x_k\|^2)\left(\frac{2}{\alpha}\right)^n\\
&\underset{k\to+\infty}{\longrightarrow} 0 \numberthis \label{prive_b}
\end{align*}Then, for any $\ph$ that is $1-$Lipschitz on the ball, we have for all $\alpha>0$ that \begin{align*}&\left|\int_{\mathbb{S}^{n-1}}\ph(y)\frac{1-\|x_k\|^2}{\|x_k-y\|^n} \mathrm{d}\sigma^{n-1}(y)- \ph (x)\right| \\ &=\left|\int_{\mathbb{S}^{n-1}}(\ph(y)-\ph(x))\frac{1-\|x_k\|^2}{\|x_k-y\|^n}  \mathrm{d}\sigma^{n-1}(y)\right| \\
&\le\left|\int_{\mathbb{S}^{n-1}\setminus B(x,\alpha)}\|y-x\|\frac{1-\|x_k\|^2}{\|x_k-y\|^n} \mathrm{d}\sigma^{n-1}(y)\right|\\ & \ \ \ \ \ \ \ \ \ \ \ \ \ +\left|\int_{B(x,\alpha)}\|y-x\|\frac{1-\|x_k\|^2}{\|x_k-y\|^n}  \mathrm{d}\sigma^{n-1}(y)\right| \\
&\le\left|\int_{\mathbb{S}^{n-1}\setminus B(x,\alpha)}2\times\frac{1-\|x_k\|^2}{\|x_k-y\|^n} \mathrm{d}\sigma^{n-1}(y)\right|+\left|\int_{B(x,\alpha)}\alpha\frac{1-\|x_k\|^2}{\|x_k-y\|^n}  \mathrm{d}\sigma^{n-1}(y)\right|.
\end{align*} The second term is inferior to $\alpha$ and using \eqref{prive_b}, we know that the first term converges to $0$. Since this is true for any $\ph$ that is $1$-Lipschitz, which are convergence determining, we have eventually that $\Pa(x_k,\, . \,) \to_k \Pa(x,\, . \,)$ in distribution.

\subsection{Explicit expression of the radial derivative}
For all $\ph\in \mathcal{C}_c(\mathbb{S}^{n-1})$  and $a\in \mathbb{S}^{n-1}$, we write $f_a:\lambda\in]0,1[\longrightarrow \Pa \varphi(\lambda a)$. We have therefore for all $z\in \mathbb{B}^n\setminus \{0\}$, writing $z=\lambda a$ with $a\in \mathbb{S}^{n-1}$ and $\lambda=\|z\|$ that $\frac{\partial}{\partial r} \Pa \varphi(z)=\frac{\partial}{\partial \lambda}f_a(\lambda)$. Let us compute this quantity:
\begin{align*}
    &\frac{\partial}{\partial \lambda} f_a(\lambda) \\
    &=\frac{\partial}{\partial \lambda} (f_a-\varphi(a))(\lambda)
    \\&=\frac{\partial}{\partial \lambda}\int_{\mathbb{S}^{n-1}} \frac{1-\lambda^2}{\|\lambda a-y\|^n}\left(\ph(y)-\ph(a) \right)\mathrm{d}\sigma^{n-1}(y) \\
    &=\frac{\partial}{\partial \lambda}\int_{\mathbb{S}^{n-1}} \frac{1-\lambda^2}{\|\lambda a-y\|^n}\left(\ph(y)-\ph(a)-\langle T\ph(a), y-a\rangle \right)\mathrm{d}\sigma^{n-1}(y) \\
    & \ \ \ \ \ \ \ \ \ \ \ +\frac{\partial}{\partial \lambda}\int_{\mathbb{S}^{n-1}} \frac{1-\lambda^2}{\|\lambda a-y\|^n}\langle T\ph(a), y-a\rangle \mathrm{d}\sigma^{n-1}(y) \numberthis \label{deriv}
\end{align*}
We notice for $y\in \mathbb{S}^{n-1}$ that $2\langle a,y\rangle a -y \in \mathbb{S}^{n-1}$, that $\|\lambda a-y\|=\|\lambda a-(2\langle a, y \rangle a-y)\|$ and, using $ T\ph(a)\perp a$, that $\langle T\ph(a),\lambda a -y  \rangle=-\langle T\ph(a),\lambda a-(2\langle a,y\rangle a -y)\rangle$ (all these properties are trivial by symmetry). We conclude that the second term in \eqref{deriv} is the derivative of $0$ and is hence null. Therefore: 
\begin{align*}
    &\frac{\partial}{\partial \lambda} f_a(\lambda) \\
    &=\frac{\partial}{\partial \lambda}\int_{\mathbb{S}^{n-1}} \frac{1-\lambda^2}{\|\lambda a-y\|^n}\left(\ph(y)-\ph(a)-\langle T\ph(a), y-a\rangle \right)\mathrm{d}\sigma^{n-1}(y) \\
    &=\int_{\mathbb{S}^{n-1}} \frac{-2\lambda \|\lambda a-y\|^n-(1-\lambda^2) n\frac{\langle\lambda a-y, a\rangle}{\|\lambda a-y\|}\|\lambda a-y\|^{n-1}}{\|\lambda a-y\|^{2n}} & \\ &\ \ \ \ \ \ \   \ \ \ \ \ \ \ \ \ \ \times \left(\ph(y)-\ph(a)-\langle T\ph(a), y-a\rangle \right)\mathrm{d}\sigma^{n-1}(y) \\
    &=\int_{\mathbb{S}^{n-1}} \frac{-2\lambda -(1-\lambda^2) n\frac{\langle\lambda a-y,a\rangle}{\|\lambda a-y\|}\|\lambda a-y\|^{-1}}{\|\lambda a-y\|^{n}}  \\ &\ \ \ \ \ \ \   \ \ \ \ \ \ \ \ \ \ \times \left(\ph(y)-\ph(a)-\langle T\ph(a), y-a\rangle \right)\mathrm{d}\sigma^{n-1}(y) 
\end{align*}
Therefore, for $z\in \mathbb{B}^n$, we have:
\begin{align*}\frac{\partial}{\partial r}  &\Pa \varphi(z)= \int_{\mathbb{S}^{n-1}} \frac{-2\|z\| -(1-\|z\|^2)n\frac{\langle z-y,a\rangle}{\| z-y\|}\|z-y\|^{-1}}{\|z-y\|^{n}} & \\ &\ \ \ \ \ \ \   \ \ \ \ \ \ \ \ \ \ \times \left(\ph(y)-\ph(\frac{z}{\|z\|})-\langle T\ph(\frac{z}{\|z\|}), y-\frac{z}{\|z\|}\rangle \right)\mathrm{d}\sigma^{n-1}(y). \end{align*}
Taking the limit when $z$ approaches $a$ (so in particular $\|z\|$ approaches $1$) we find that:

\begin{align*}&\lim_{z\to a} \frac{\partial}{\partial r} \Pa \varphi(z)  \\ &= -2\lim_{z\to a}\|z\|\int_{\mathbb{S}^{n-1}} \frac{1}{\|z-y\|^n} \left(\ph(y)-\ph(\frac{z}{\|z\|})-\langle T\ph(\frac{z}{\|z\|}), y-\frac{z}{\|z\|}\rangle \right)\mathrm{d}\sigma^{n-1}(y) \\
& \ \ \ \ \ \ \ -\lim_{z\to a}(1-\|z\|^2)\int_{\mathbb{S}^{n-1}} \frac{n\frac{\langle z-y,a\rangle}{\| z-y\|}\|z-y\|^{-1}}{\|z-y\|^{n}} \\ &\ \ \ \ \ \ \   \ \ \ \  \ \ \ \ \ \ \ \ \ \ \ \times \left(\ph(y)-\ph(\frac{z}{\|z\|})-\langle T\ph(\frac{z}{\|z\|}), y-\frac{z}{\|z\|}\rangle \right)\mathrm{d}\sigma^{n-1}(y). \numberthis \label{prem}
\end{align*}
For the first term in \eqref{prem}, we observe that the integrand converges simply to what we are looking for when $z$ approaches $a$. We now denote $R_z$ for $z\in \mathbb{B}\setminus \{0\}$ the only rotation sending $\frac{z}{\|z\|}$ on $a$. Then, using that the measure $\sigma$ is invariant by rotations and that rotations are isometries, we have that:
\begin{align*}
    &\int_{\mathbb{S}^{n-1}} \frac{1}{\|z-y\|^n} \left(\ph(y)-\ph(\frac{z}{\|z\|})-\langle T\ph(\frac{z}{\|z\|}), y-\frac{z}{\|z\|}\rangle \right)\mathrm{d}\sigma^{n-1}(y) \\
    &=\int_{\mathbb{S}^{n-1}} \frac{1}{\|R_z(z)-y\|^n} \left(\ph(R_z^{-1}(y))-\ph(\frac{z}{\|z\|})\right. \\
    & \ \ \ \ \ \ \ \ \ \ \ \ \ \ \ \ \ \ \ \ \ \ \ \ \ \ \ \ \ \ \ \ \ \ \ \  \left. -\langle T\ph(\frac{z}{\|z\|}), R_z^{-1}(y)-\frac{z}{\|z\|}\rangle \right)\mathrm{d}\sigma^{n-1}(y) \numberthis \label{rot}
\end{align*}
Since $\varphi$ in in $\mathcal{C}^2(\mathbb{S})$, there exists $M>0$ such that:$$\left|\left(\ph(R_z^{-1}(y))-\ph(\frac{z}{\|z\|})-\langle T\ph(\frac{z}{\|z\|}), R_z^{-1}(y)-\frac{z}{\|z\|}\rangle \right)\right|\leq M\|y-a\|^2.$$
Furthermore $R_z(z)=\|z\|a$, so eventually the integrand in \eqref{rot} is dominated by $2^nM\|y-a\|^{-(n-2)}$ that is integrable on the sphere, and we can use the dominated convergence theorem to conclude.

Let us now prove that the second term in \eqref{prem} is equal to $0$. We notice that $|\frac{\langle z-y,a\rangle}{\| z-y\|}|\leq 1$ and that $\ph(y)-\ph(\frac{z}{\|z\|})-\langle T\ph(\frac{z}{\|z\|}), y-\frac{z}{\|z\|}\rangle$ is dominated by $4M\|z-y\|^2$, so eventually the term inside the integral is dominated by $4nM\|z-y\|^{-(n-1)}$. Using a polar change of coordinates, we get that the integral on the sphere is hence dominated by $4nM|\ln(\|z-\frac{z}{\|z\|}\|)|=4nM|\ln(1-\|z\|)|$, but $$(1-\|z\|^2)|\ln(1-\|z\|)|\underset{z\to a}{\longrightarrow} 0,$$ so we have \eqref{exp_Levy}.

\end{appendices}

\section*{Acknowledgements} The authors want to thank Rapha\"el Ch\'etrite and C\'edric Bernardin for their help and support. This work has been partially supported by ANR SINEQ, ANR-21-CE40-0006.

\bibliographystyle{plain}
\bibliography{refs,chetrite}

\begin{thebibliography}{10}

\bibitem{q3}
M.~Gregoratti~A. Barchielli.
\newblock {\em Quantum trajectories and measurements in continuous time: the
  diffusive case}.
\newblock Springer, Verlag, 2009.

\bibitem{BBCCNP}
T.~Benoist, C.~Bernardin, R.~Chetrite, R.~Chhaibi, J.~Najnudel, and C..
  Pellegrini.
\newblock Emergence of jumps in quantum trajectories via homogenization.
\newblock {\em Commun. Math}, 2021.

\bibitem{borodin2015handbook}
A.N. Borodin and P.~Salminen.
\newblock {\em Handbook of Brownian Motion - Facts and Formulae}.
\newblock Probability and Its Applications. Birkh{\"a}user Basel, 2015.

\bibitem{q1}
F.~Petrucionne H.~P. Breuer.
\newblock {\em The Theory of Open Quantum Systems}.
\newblock Oxford University Press, 2002.

\bibitem{cedric2022spike}
Bernardin C{\'e}dric, Chhaibi Reda, Najnudel Joseph, and Pellegrini
  Cl{\'e}ment.
\newblock To spike or not to spike: the whims of the wonham filter in the
  strong noise regime.
\newblock {\em arXiv preprint arXiv:2211.02032}, 2022.

\bibitem{dellacherie1975}
C.~Dellacherie and P.A. Meyer.
\newblock {\em Probabilit{\'e}s et potentiel, vol.~1}.
\newblock Hermann, 1975.

\bibitem{doob2012classical}
J.L. Doob.
\newblock {\em Classical Potential Theory and Its Probabilistic Counterpart}.
\newblock Classics in Mathematics. Springer Berlin Heidelberg, 2012.

\bibitem{dudley2018real}
Richard~M Dudley.
\newblock {\em Real analysis and probability}.
\newblock CRC Press, 2018.

\bibitem{EK}
N.~Ethier and T.~G. Kurtz.
\newblock {\em {M}arkov Processes : Characterizations and Convergence}.
\newblock Wiley, Interscience, 2005.

\bibitem{evans2010partial}
L.C. Evans.
\newblock {\em Partial Differential Equations}.
\newblock Graduate studies in mathematics. American Mathematical Society, 2010.

\bibitem{faure2022averaging}
Dimitri Faure.
\newblock Averaging of semigroups associated to diffusion processes on a
  simplex.
\newblock {\em Stochastic Processes and their Applications}, 150:323--357,
  2022.

\bibitem{q2}
J.-M. Raimond~S. Haroche.
\newblock {\em Exploring the Quantum: Atoms, CaviQes, and Photons}.
\newblock Graduate Texts, Oxford, 2006.

\bibitem{hsu1986excursions}
Pei Hsu.
\newblock On excursions of reflecting brownian motion.
\newblock {\em Transactions of the American Mathematical Society},
  296(1):239--264, 1986.

\bibitem{kechris2012classical}
A.~Kechris.
\newblock {\em Classical Descriptive Set Theory}.
\newblock Graduate Texts in Mathematics. Springer New York, 2012.

\bibitem{kurtz2011}
Thomas~G. Kurtz.
\newblock {\em Equivalence of Stochastic Equations and Martingale Problems},
  pages 113--130.
\newblock Springer Berlin Heidelberg, Berlin, Heidelberg, 2011.

\bibitem{MZ84}
P.~A. Meyer and W.~A Zheng.
\newblock Tightness criteria for laws of semimartingales.
\newblock {\em Ann. I. H. Poincaré - Pr}, 20(4):353--372, 1984.

\bibitem{k12}
G.~C. Papanicolaou, D.~Stroock, and S.~R.~S. Varadhan.
\newblock Martingale approach to some limit theorems.
\newblock In N.~C. Durham, editor, {\em Proceedings of the 1976 Duke University
  Conference on Turbulence}, 1976.

\bibitem{pardoux2}
{\'E}tienne Pardoux.
\newblock Homogenization of linear and semilinear second order parabolic pdes
  with periodic coefficients: a probabilistic approach.
\newblock {\em Journal of Functional Analysis}, 167(2):498--520, 1999.

\bibitem{pardoux1}
E.~Pardouxd and A.Yu Veretennikov.
\newblock Averaging of backward stochastic differential equations, with
  application to semi-linear pde's.
\newblock {\em Stochastics and Stochastic Reports}, 60(3-4):255--270, 1997.

\bibitem{pilipenko2014introduction}
Andrey Pilipenko.
\newblock {\em An introduction to stochastic differential equations with
  reflection}, volume~1.
\newblock Universit{\"a}tsverlag Potsdam, 2014.

\bibitem{sylvester1987global}
John Sylvester and Gunther Uhlmann.
\newblock A global uniqueness theorem for an inverse boundary value problem.
\newblock {\em Annals of mathematics}, pages 153--169, 1987.

\bibitem{williams1979diffusions}
D.~Williams, L.C. Rogers, and L.C.G. Rogers.
\newblock {\em Diffusions, Markov Processes, and Martingales}.
\newblock Number vol.~2 in Cambridge mathematical library. Wiley, 1979.

\bibitem{q4}
G.~J. Milburn H.~M. Wiseman.
\newblock {\em Quantum Measurement and Control}.
\newblock Cambridge University Press, 2009.

\end{thebibliography}
\end{document}